\documentclass[12pt]{article}

\usepackage{lineno,hyperref,color}
\usepackage{epsfig}
\usepackage{graphicx}
\modulolinenumbers[5]
\usepackage{amsfonts,amstext,multirow}
\usepackage[english,activeacute]{babel}
\usepackage[applemac]{inputenc} 
\usepackage{graphics,graphicx}
\usepackage{hyperref}
\usepackage{soul} 
\usepackage{cancel} 
\usepackage{color}

\usepackage{epstopdf}
\usepackage[colorinlistoftodos]{todonotes}
\usepackage{psfrag} 
\usepackage{amsthm}
\usepackage{amsmath, amssymb, dsfont}
\usepackage{bm}
\usepackage{ulem}
\usepackage{geometry}
\newtheorem{theorem}{Theorem}[section]

\newtheorem{proposition}[theorem]{Proposition}

\newtheorem{remark}[theorem]{Remark}

\newcommand\redout{\bgroup\markoverwith{\textcolor{red}{\rule[.5ex]{2pt}{0.8pt}}}\ULon}
\usepackage{color}
\newcommand{\coln}{\color{black}}

\def\trh{{\cal T}_{h}}
\def\disp{\displaystyle}

\def\a{{\boldsymbol a}}
\def\beq{\begin{equation}}
\def\eeq{\end{equation}}

\def\R{\mathbf{R}}
\def\P{\mathbb{P}}
\def\X{\mathbf{X}}

\usepackage{listings}
\begin{document}
\title{Least-squares driven stabilised finite element solution of advection-dominated flow problems}
\author{Tom\'as Chac\'on Rebollo \thanks{Dpto. EDAN\&IMUS, University of Seville, Campus de Reina Mercedes, 41012 Sevilla, Spain}, Daniel Franco Coronil \thanks{Dpto. EDAN, University of Seville, Campus de Reina Mercedes, 41012 Sevilla, Spain}}

 \maketitle
\begin{abstract}
In this article, we  address the solution of advection-dominated flow problems by stabilised methods, by means of least-squares computed stabilised coefficients. As main methodological tool, we introduce a data-driven off-line/on-line strategy to compute them with low computational cost. 

We compare the errors provided by the least-squares stabilised coefficients to those provided by several previously established stabilised coefficients within the solution of advection-diffusion and Navier-Stokes flows, on structured and un-structured grids, with and Lagrange Finite Elements up to third degree of interpolation. In all tested flows the least-squares stabilised coefficients provide quasi-optimal errors.

We conclude that the least-squares procedure is a rewarding procedure, worth to be applied to general stabilised solutions of general flow problems.

\vspace*{2ex}\par

\end{abstract}



\section{Introduction} \label{sec:intro}
Stabilised methods provide a general technique to treat the instabilities generated by the Galerkin discretisation of partial differential equations (PDEs). Such instabilities are due to low-order derivation terms appearing in the PDEs, that become dominant at discrete level when the discretisation parameters are not small enough. This generates spurious oscillation in the discrete solutions that therefore are unreliable for practical applications (see Hughes (cf. \cite{Hughes_1995, HughesStewart_1995, HughesFMQ_1998}).  The initial stabilised method is the SUPG (Streamline Upwind Petrov-Galerkin) method of Brooks and Hughes (see \cite{Brooks_1982}), that consists in adding an extra term to the Galerkin discretization to control a weak norm of the advection derivative. Several classes of stabilised methods followed the SUPG one (Galerkin-Least Squares, Adjoint Stabilised, Orthogonal Sub-Scales methods, among others). All of them consisted in adding residual-based extra terms to the Galerkin formulation to control the low-order terms that appear in the PDEs. These methods were successively applied to incompressible and compressible flow equations, additionally providing a further stabilisation of the pressure gradient discretisation (cf. \cite{Hughes_1995} for an overview).
\par
 A particular class of stabilised methods is the Variational Multi-Scale (VMS) one, based upon the modelling of the sub-grid scales effect on the resolved scales. This requires to approximately solve the small scale problem in terms of the resolved scales, and plugging the resolved approximate small scales in the resolved scales equation. This provides improved stable and acccurate solutions to many flow problems (cf.\cite{hug112, john0612, Chaconlibro}). 
  \par
Solving the sub-grid scales by approximate diagonalisation of the PDE operator acting on them leads to the Orthogonal Sub-Scales (OSS) method, introduced by Codina in \cite{codina12}. A further step within the diagonalization techniques is the use of spectral techniques, introduced in \cite{ChaconDia}, that leads to the spectral VMS method (cf. \cite{ChaconDia, solfergar}) .
\par 
A further simplified stabilised method is the term-by-term one, in which specific least-squares terms are added to stabilise each actual low-order operator term that could generate instabilities (Cf. \cite{Chacon, macaisatomas}). 

All stabilised methods, in addition to their structure, are characterised by by \lq \lq stabilised coe\-ffi\-cients" that govern the element-wise stabilisation strength of the additional terms. For flow equations, this strength depends on the relative size of diffusion and advection terms at element level. The accuracy, as well as the stabilising properties of the method, largely depend on the actual expression of these stabilised coefficients. {A number of stabilisation parameters have been considered so far in the literature. For 1D flows, there exist optimal coefficients that ensure that the discrete solutions coincides with the exact one at grid nodes (cf. \cite{christie,johnkno}). Codina introduced in \cite{codina12} a formula for stabilization coefficients for multi-dimensional flows based upon orthogonal sub-grid scales. This formula was extended to an anisotropic version considering the streamline length of the element by Codina and Colom\'es in \cite{Colomes2018}. Also, in \cite{Franca2000} Franca and Valentin had introduced an anisotropic stability coefficient, observing that it yields the best numerical results is computed using the largest streamline length of the element.  Such coefficient is only valid for negative source terms, it was extended to positive source terms by Hauke in \cite{Hauke2002}, and used afterwards, for instance, in \cite{Hauke2008}. 
\par
In the present paper, we afford the computation of the stabilised parameters by least-squares techniques for Lagrange finite element solutions of 1D and 2D advection-diffusion flows. We minimise the quadratic distance between the stabilised solution and the Lagrange interpolate of a reference solution at Lagrange interpolation nodes. This turns out to the 1D minimisation of a smooth convex functional, that is readily carried on by standard techniques. We introduce a data-driven off-line/on-line strategy to compute the stabilised coefficients, considered as functions of the non-dimensional parameters (actually, the directional grid P\'eclet numbers), that govern the flow at grid element level. In the off-line stage, the stabilised coefficients are computed at the nodes of grid of the directional grid P\'eclet numbers. In the on-line step the stabilised coefficients on each grid element are computed by a fast interpolation procedure of the pre-computed values at the grid P\'eclet number nodes, thus requiring quite reduced computation times.

We perform some numerical tests to compare the errors provided by the least-squares stabilised coefficients to those provided by known stabilised coefficients, with several advection-diffusion and Navier-Stokes flows, considering isotropic and anisotropic advection velocities, as well as isotropic and anisotropic grids, and $\P_1$, $\P_1$+Bubble, $\P_2$ and $\P_3$ finite elements.  We observe that in all tested flows the least-squares stabilised coefficients provide either the smallest errors, or error levels quite close to the smallest ones. 

The least-squares procedure to compute the stabilised coefficients is a general procedure, that can be applied to finite element, finite volume or spectral discretisations, as well as to more general (compressible, multi-phase, thermal, ...) fluid flows.  In despite of its need of a rather large amount of computation in the off-line procedure, it appears as a rewarding procedure, worth to be applied to general stabilised solutions of flow problems.
\par
The paper is outlined as follows. In Section 2 we describe the off-line/on-line strategy to compute the stabilised coefficients, while in Section 3 we describe the least-squares procedure to compute them. In Section 4 we present the numerical tests. We address some conclusions and perspectives of future work in Section 5.
}
%
%
\section{Off-line/on-line strategy to compute the stabilised coefficients} \label{se:least}
In this section we describe the off-line/on-line strategy that we follow to compute the stabilised coefficients. With this purpose, se consider the advection-diffusion equations,
\begin{equation}\label{WEAD}
\left \{\begin{array}{l}
\mbox{Find a passive scalar }u:\bar{\Omega}\mapsto \R \quad\mbox{such that}\\  
\begin{array}{rcl}
\a\cdot\nabla u+\nabla\cdot (\mu \,\nabla u)&=&f \,\,\mbox{in } \Omega,  
\\
u&=&0 \,\,\mbox{on } \partial\Omega
\end{array}
\end{array}
\right .
\end{equation}
where $\Omega \subset \R^d$  is a bounded domain (with $d=1$, $d=2$ or $3$); $\a \in L^\infty(\Omega)^d$ is the advection velocity that we assume to be divergence-free;  $\nabla \cdot \a=0$, $\mu \in L^\infty(\Omega)$ is the viscosity, that verifies $\mu(x)\ge \mu_0 >0$ a. e. in $\Omega$; and $f \in L^2(\Omega)$ isthe forcing terms. We consider homogeneous Dirichlet boundary conditions for simplicity although our methodology can readily be extended to Neumann or mixed boundary conditions. We consider the standard variational formulation of these equations, 
\begin{equation}\label{WEAD}
\left\{\begin{array}{l}
\mbox{Find }u\in H_0^1(\Omega) \quad\mbox{such that,}\\  \noalign{\smallskip}
(\a\cdot\nabla u,v)+(\mu \,\nabla u,\nabla v)=(f,v) \quad \forall v\in H_0^1(\Omega).\\  \noalign{\smallskip}
\end{array}\right.
\end{equation}
To establish the stabilised methods that we are considering, let us assume that $\Omega$ is either an interval (when $d=1$), polygon (when $d=2$) or polyhedron (when $d=3$), and consider a family of triangulations $\{\trh\}_{h>0}$ of $\Omega$, that we assume regular in the sense of Ciarlet \cite{Ciarlet}. We consider the conformal Lagrange finite element spaces 
$$
V_h^{(l)}=\{v_h \in C^0(\Omega)\,|\, {v_h}_{|_K} \in \P_l(K),\,\, \forall \, K \in \trh \,\},\,\, V_{0h}^{(l)}=V_h^{(l)} \cup H^1_0(\Omega),
$$
where $\P_l(K)$ denotes the space of polynomials of degree less or equal than $l$ defined on $K$. Let us consider the following discrete spaces for passive scalar, velocity and pressure, respectively:
$$
X_h=V_{0h}^{(l)},\quad \X_h^{(l)}=[V_{0h}^{(l)}]^d,\quad M_h= V_h^{(l)}.
$$
We initially consider the following stabilised discretisations of problem \eqref{WEAD},

Find $u_h \in X_h$ such that
\begin{equation} \label{disad}
a(u_h,v_h)+ s_h(P(u_h),Q(v_h))=  \langle \tilde{f},v_h\rangle  \quad \forall v_h \in X_h,
\end{equation}
where 
$$
a(u,v)= (\a\cdot\nabla u,v)+(\mu \,\nabla u,\nabla v),\quad \forall u,\, v \in H^1_0(\Omega);
$$
\begin{equation} \label{eseh}
s_h(\alpha,\beta)=\sum_{K \in \trh} \tau_K\, (\alpha,\beta)_K,\quad \forall \alpha,\, \beta \in L^2(\Omega),
\end{equation}
$$
\langle \tilde{f},v_h\rangle= ( f,v_h)  + s_h(f,Q(v_h)), \quad \forall v_h \in X_h,
$$
$$
P(u_h)_{|_K}=\a_{|_K}\cdot \nabla {u_h}_{|_K}-\nu\, \Delta {u_h}_{|_K},\quad
Q(v_h)_{|_K} =\a_{|_K}\nabla\cdot {u_h}_{|_K}+\varepsilon\,\nu\, \Delta {u_h}_{|_K},\,\, \forall K \in \trh,
$$
where $\varepsilon$ may take the values $-1$, $0$ or $1$, respectively corresponding to the Least-Squares, Streamline Upwind and Adjoint stabilised methods. We also consider the term-by-term stabilised method, that corresponds to
\begin{equation} \label{tbt}
P(u_h)_{|_K}=\a_{|_K}\cdot \nabla {u_h}_{|_K},\quad
Q(v_h)_{|_K} =\a_{|_K}\cdot \nabla{u_h}_{|_K},\,\, \forall K \in \trh,
\end{equation}
and
$$
\langle \tilde{f},v_h\rangle= ( f,v_h)  , \quad \forall v_h \in X_h. 
$$
In \eqref{eseh}, $\tau_K$ are the stabilised coefficients, that determine each actual stabilised method, once $P$ and $Q$ have been set. When $d=1$ and $\a$, $f$ are constant, there exists an optimal setting of the $\tau_K$ that ensures that $u_h$ coincides with $u$ at the Lagrange interpolation nodes (cf. \cite{christie,johnkno}). This is a consequence of the decomposition
$$
H^1(\Omega)= V_h^{(l)} \bigoplus_{K\in \trh} H^1_0(K).
$$
However this is no longer true in higher dimensions. In this case if the stabilised coefficients take into account the advection-dominated regime, it can be proved (cf. \cite{chacondomregalo}) that the seminorm of the advection derivative,
$$
\left (\sum_{K \in \trh} \tau_K\, \|\a\cdot \nabla {u_h}\|_{0,K}^2\right)^2
$$
is uniformly bounded independently of $\mu_0$ and $h$. We here afford to minimise the difference $u_h-u$ at the Lagrange interpolation nodes with an optimal choice of the stabilisation parameters.  

\subsection{Offline/online strategy}
We address the computation of universal formulas for the stabilised coefficients, by means of an offline/online strategy based upon the use of data-driven least squares techniques. Of course this could be done for each actual $\a$, $\mu$ and $f$ but this would yield an useless method. We start from the observation that by dimensional analysis, $\tau_K$ has the dimension of a time. If $h_K$ is the element diameter, then $\varphi_K=\displaystyle \frac{\|\a\|_K}{h_K}\, \tau_K$ is dimensionless, where $\|\a\|_K$ denotes some norm of the velocity $\a$ on element $K$. We assume that $\varphi_K$ only depends on the dimensionless parameters that determine the flow on element $K$, which are the direction element P\'eclet numbers,
$
\disp P_{iK}= \frac{\bar{a}_{i_K}\, h_K}{2 \mu_K},\,\, i=1,\cdots, d,
$
where $\bar{a}_{i_K}$ is the average of the i-th component of the velocity on $K$ and $\mu_K$ is an average value of $\mu$ on $K$. That is,
$$
\varphi_K=\varphi(P_{1K}, \cdots,P_{dK})
$$
for some function $\varphi(P_1,\cdots,P_d)$ of the real variables $P_1,\cdots,P_d$. Thus knowing $\varphi$ we may compute the local stabilisation coefficients by
\begin{equation} \label{taufi}
\tau_K=\displaystyle \frac{h_K}{\|a\|_K}\,\varphi(P_{1K}, \cdots,P_{dK}).
\end{equation}
Of course this is just an approximation, as actually the P\'eclet numbers do vary in space, we are just taking average values.

For evolution advection-diffusion equations, the function $\varphi$ would also depend on the \lq\lq time advection" P\'eclet number, $
\disp S_K= \frac{\Delta t\,\mu_K}{h_K^2 }$. Here we just consider 1D and 2D steady problems as a first approach to the least-squares computing of the stabilised coefficients.

Our strategy is to compute the function $\varphi(P_{1}, \cdots,P_{d})$ on the nodes of a grid $\cal G$ within a parallelepiped $R= [0,{\cal P}_1]\times \cdots \times [0,{\cal P}_d] \subset \R^d$ in an off-line step. We set
$$
{\cal G}= \{ \alpha_{i_1,\cdots,i_d}= (\Delta_1 \, i_1 ,\cdots, \Delta_d \, i_d)\in \R^d, \,\,\mbox{for  }\, i_1=0,1,\cdots, M_1,\cdots, i_d=0,1,\cdots, M_d\,\},
$$
with 
$$
\Delta_i = {\cal P}_i/M_i\,\quad\mbox{for  } i=1,\cdots, d
$$
for some positive integer numbers $M_1,\cdots,M_d$. The upper extremes ${\cal P}_i$ of the intervals in $R$ are chosen in such a way that $\varphi(P_{1}, \cdots,P_{d})$ becomes nearly constant as the variable $P_i$ approaches ${\cal P}_i$, as indeed occurs in practice for any $i=1,\cdots, d$. We use a least squares technique to compute the values $\varphi(\alpha_{i_1,\cdots,i_d})$, described in the next section. 

In the online step, for a given set of values $(P_{1K}, \cdots,P_{dK})$ we determine the indices $i_1,\cdots,i_d$ such that $(P_{1K}, \cdots,P_{dK}) \in [\alpha_{i_1,\cdots,i_d}, \alpha_{i_1+1,\cdots,i_d}]\times\cdots\times [\alpha_{i_1,\cdots,i_d}, \alpha_{i_1,\cdots,i_d+1}]$ and compute $\varphi(P_{1K}, \cdots,P_{dK})$ by second order interpolation of the computed values of $\varphi$ at these nodes. Then the stabilised coefficient $\tau_K$ is computed by \eqref{taufi}.

This procedure has the advantage to apply to any kind of stabilised method, as well as to any $H^1$-conformal finite element space.

\section{Least-squares off-line computation of stabilised coefficients}\label{se:least}
Let us consider a generic consistent stabilised method, with the structure \eqref{disad}.
We determine the stabilised coefficients by comparison of the solution $u_h$ provided the method \eqref{disad} with a high-fidelity solution $u \in H_0^1(\Omega)$, as in general we do not have the actual analytic solution of problem \eqref{WEAD}. Typically this solution is obtained by solving \eqref{disad} with some reference stabilised coefficients, on a much finer grid than $\trh$. 

Note that as $\tau$ depends on the local P\'eclet numbers $P_{iK}$, if the velocity $\a$, the diffusion $\mu$ and the grid size $h_K$ are constant, all $\tau_K$ are equal to a value $\tau$.  We thus consider this situation, and address the problem of determining this value $\tau$ as a function of the element-independent local P\'eclet numbers $P_i=\disp \frac{a_i\, h}{2\mu}$. This will provide $\varphi$ as a function of $P=(P_1,\cdots,P_d)$ in the general case.

We then search for $\tilde{\tau} =\tau(P)$ solution of 
\begin{equation}\label{optau}
\tilde{\tau} =\mbox{argmin}\{J(\tau),\, \tau \in [\tau_{min},\tau_{max}]\},\quad \mbox{with  }\,\, 
\end{equation}
$$
J(\tau)= \frac{1}{2}\,\|u_h(\tau)-\Pi_h(u)\|^2_0,
$$
where $\Pi_h$ is the Lagrange interpolation operator on space $X_h$, $u_h(\tau)$ is the solution of problem \eqref{WEAD} for $\tau_K=\tau$ for all $K\in \trh$ and $\tau_{min}$, $\tau_{max}$ is the estimated minimum and maximum values that can reach the $\tau_K$, we actually set 
$$
\tau_{min}=\alpha_{min}\, h^2,\quad \tau_{max}=\alpha_{max}\, h^2
$$ 
for some $\alpha_{max}>\alpha_{min}>0$, as the standard expressions for $\tau$ imply that it is of order $h^2$ . Once  $\tilde{\tau}$ is computed, we set
$$
\varphi(P)=\frac{\|a\|}{h}\, \tilde{\tau}
$$
In this way we are minimising the error between the solution of the stabilised method and the exact solution at grid nodes. Let us recall that for 1D advection-diffusion problems with constant data, this minimum is zero for the optimal stabilisation coefficients (cf. \cite{}). Then it makes sense to target to minimise the error with respect to the Lagrange interpolate of the exact solution rather than to the exact solution itself.  

Concerning the existence of solution of problem \eqref{optau}, it holds
\begin{proposition}
Assume that the solution $u$ of problem \eqref{WEAD} belongs to $H^2(\Omega)$ and that $\Delta u \ne 0$, $\a \in L^\infty(\Omega)$, and the adjoint problem to \eqref{WEAD} is $H^2$-regularising. Assume also that the discrete problem \ref{disad} is the term-by-term stabilised method  given by \eqref{tbt}. Then problem \eqref{optau} admits a unique solution when $\alpha_{min}$ is large enough, for small enough $h$.
\end{proposition}
\begin{proof} The functional $J$ is continuous as $u_h(\tau)$ is a continuous function of $\tau$. Indeed, $u_h(\tau)$ is obtained by means of the solution of a non-singular linear systems whose coefficients are continuous functions of $\tau$. Then $J$ admits at least a minimum in $[0,\tau_{max}]$. We next prove that $J$ is strictly convex to ensure the uniqueness of the minimum. It holds
$$
J'(\tau)= (u_h(\tau)-\Pi_h(u), z_h),\quad J''(\tau)= \|z_h\|_0^2 - (u_h(\tau)-\Pi_h(u), w_h)
$$
where $z_h=\disp \frac{du_h}{d\tau}(\tau),\,\, w_h=\frac{d^2u_h}{d\tau^2}(\tau) \in X_h$ are the solutions of
\beq \label{pbzh}
a(z_h,v_h) + \tau \, (\a\cdot \nabla z_h, \a\cdot \nabla v_h) = (f- \a\cdot \nabla u_h, \a\cdot \nabla v_h),\quad \forall v_h \in X_h,
\eeq
\beq \label{pbwh}
a(w_h,v_h) + \tau \, (\a\cdot \nabla w_h, \a\cdot \nabla v_h) = -2\,( \a\cdot \nabla z_h, \a\cdot \nabla v_h),\quad \forall v_h \in X_h,
\eeq
As the adjoint problem to \eqref{WEAD} is $H^2$-regularising, then (cf. \cite{})
\beq \label{errpih}
\|u_h(\tau)-\Pi_h(u)\|_0 \le C\, h^{l+1} \,\|u\|_{2,2,\Omega}.
\eeq
Also, setting $v_h=w_h$ in \eqref{pbwh}, 
$$
\mu \, \|\nabla w_h\|_0^2 + \tau\,\|\a\cdot \nabla w_h\|_0^2 = -2\,( \a\cdot \nabla z_h, \a\cdot \nabla w_h)
$$
Applying Young's inequality, it follows
\begin{eqnarray} \label{est1wh}
\mu \, \|\nabla w_h\|_0^2 + \frac{\tau}{2}\,\|\a\cdot \nabla w_h\|_0^2 &\le&  2\,\tau^{-1}\, \|\a\cdot \nabla z_h\|_0^2  \le 2\,\tau^{-1}\, \|\a\|_{\infty,\Omega}^2\, \sum_{K\in \trh} \|\nabla z_h\|_{0,K}^2 \nonumber\\
&\le& C\,\tau^{-1}\, \|\a\|_{\infty,\Omega}^2\, h^{-2} \|z_h\|_0^2,
\end{eqnarray}
where in the last inequality we have applied the inverse finite element estimate 
$$
\|\nabla z_h\|_{0,K} \le C\, h^{-1} \, \|z_h\|_{0,K},
$$
for some constant $C>0$ depending only on $\Omega$. Then,
\begin{eqnarray}\label{estwh}
 \|\nabla w_h\|_0 \le C\, \|\a\|_{\infty,\Omega}\,\mu^{-1/2}\, \tau^{-1/2}\, h^{-1}\, \|z_h\|_0.
\end{eqnarray}
Then, from \eqref{errpih} and \eqref{estwh}, it follows 
\beq\label{estprod1}
(u_h(\tau)-\Pi_h(u), w_h)\le C\, \|\a\|_{\infty,\Omega}\,\|u\|_{2,2,\Omega}\,\mu^{-1/2}\, \tau^{-1/2} \, h^l\, \|z_h\|_0.
\eeq
Then,
$$
J''(\tau) \ge \|z_h\|_0\, (\|z_h\|_0 - C'\,  \tau^{-1/2} \, h^l) \ge \|z_h\|_0\, (\|z_h\|_0 - C'\, \alpha_{min}^{-1/2}).
$$
for some constant $C'>0$.
Further, from \eqref{pbzh} similarly to \eqref{est1wh} it follows
\begin{eqnarray*}
\mu \, \|\nabla z_h\|_0^2 + \frac{\tau}{2}\,\|\a\cdot \nabla z_h\|_0^2 &\le&  2\,\tau^{-1}\, \|f-\a\cdot \nabla u_h\|_0^2 .
\end{eqnarray*}
Consequently, the sequence $\{z_h\}_{h>0}$ is bounded in $H^1_0(\Omega)$ and then it contains a sub-sequence weakly convergent in $H^1_0(\Omega)$ to some $z$, that we denote in the same way. Given $v \in H^1_0(\Omega)$, let us consider a sequence $\{v_h\}_{h>0}$ with $v_h \in X_h$ strongly convergent to $v$ in $H^1_0(\Omega)$. As $\tau \le C\, h^2$, then
$$
\lim_{h \to 0} \tau \, (\a\cdot \nabla z_h, \a\cdot \nabla v_h) =0,
$$
and then passing to the limit in problem \eqref{pbzh}, it follows that $z$ satisfies
$$
a(z,v)= (f- \a\cdot \nabla u, \a\cdot \nabla v)\quad \forall v \in H^1_0(\Omega).
$$
As this problems admits a unique solution, then the full sequence $\{z_h\}_{h>0}$ weakly converges to $z$. As $\Delta u \ne 0$, then $f- \a\cdot \nabla u \ne 0$ and $z $ does not vanish. Due to the weak lower continuity of the norm in Hilbert spaces, it follows that there exists $h_0>0$ such that 
$$
\|z_h\|_0 \ge \frac{1}{2}\, \|z\|_0\quad \mbox{for  } 0<h<h_0.
$$
Consequently,  if $\alpha_{min} > \disp \left ( \frac{2C'}{\|z\|_0}\right )^2$ and $0<h<h_0$,
$$
J''(\tau) \ge \frac{1}{2}\, \|z\|_0\, (\frac{1}{2}\, \|z\|_0 - C'\, \alpha_{min}^{-1/2}),
$$
and then $J''(\tau)>0 $.
\end{proof}
\begin{remark} Observe that $z$ vanishes for consistent discretisations. We may use formal arguments to have an overall understanding of why also in this case $J''(\tau) >0$ for small enough $h$. Indeed, let $Q=P$ for simplicity, the argument for general $Q$ is similar but more invoved. The derivatives $\displaystyle z_h=\frac{d u_h}{d\tau} (\tau)$, $\displaystyle w_h=\frac{d^2 u_h}{d\tau^2} (\tau)$ satisfy
\beq \label{pbzhg1}
a(z_h,v_h) + \tau \, (P (z_h),P( v_h)) = (R(u_h), P( v_h)),\quad \forall v_h \in X_h,
\eeq
\beq \label{pbzhg2}
a(w_h,v_h) + \tau \, (P (w_h),P( v_h)) = -2(P(u_h), P( v_h)),\quad \forall v_h \in X_h,
\eeq
where $R(u_h)=f-P(u_h)=P(u-u_h)$ is the residual at $u_h$. From \eqref{pbzhg1}, by standard arguments,
$$
\mu \, \|\nabla z_h\|_0^2 +\frac{ \tau }{2}\,\|P(z_h)\|_0^2 \le \frac{ \tau^{-1} }{2}\,\|R(u_h)\|_0^2,
$$
and we may conjecture that as $h \searrow 0$, $\|\nabla z_h\|_0^2$ will asymptotically behave as $ \tau^{-1}\, \|R(u_h)\|_0^2$. Consequently, $\|z_h\|_0^2$ will at least scale (in terms of $h$) with this rate as $h \searrow 0$. Also, from \eqref{pbzhg2},
$$
\mu \, \|\nabla w_h\|_0^2 +\frac{ \tau }{2}\,\|P(w_h)\|_0^2 \le 2 \, \tau^{-1} \,\|P(z_h)\|_0^2 \le 2 \, \tau^{-3} \,\|R(u_h)\|_0^2.
$$
Then,
\begin{eqnarray*}J''(\tau) &\ge& C_1\, \tau^{-1}\, \|R(u_h)\|_0^2- C_2\,\|u_h-\Pi_h u\|_0 \,\|\nabla w_h\|_0 \ge C_1\, \tau^{-1}\,  \|R(u_h)\|_0 - C_3\,\tau^{-3/2}\,\|u_h-\Pi_h u\|_0) \\
&\ge& \tau^{-1}\,h^{k-1}\, ( C_4 - C_5\, \tau^{-1/2}\,h^2),
\end{eqnarray*}
where we have used the finite element estimates
$$
\|R(u_h)\|_0= \|P(u-u_h)\|_0 \le C\, h^{k-1},\quad \|u_h-\Pi_h u\|_0\le C h^{k+1},
$$
that hold for $u \in H^{k+1}(\Omega)$. Then if $h \le \alpha_{min}^{1/2}\, C_4/C_5$, $J''(\tau) >0$ on all the interval $[\tau_{\min} ,\tau_{\max} ]$.
\end{remark}
\subsection{Numerical testing}
Figure \ref{fig1}a displays the functional J for the stabilised solution of 1D and 2D advection-diffusion equations with $\P_1$ finite elements. We may observe the convexity of $J$. In the 1D case, corresponding to $Pe=1.6667$, the functional vanishes in a unique value of $\tau$ ($\tau \simeq 0.001185$). This means that for this value the stabilised solution coincides with the exact one at the grid nodes, as we expected, so we recover the optimal stabilised coefficients. However for 2D advection-diffusion equations the minimum of the functional $J$ is not zero (see Figure \ref{fig1}b), $J$ achives its minimum at  $\tau \simeq 0.00317$ with value  $1.5733e-05$, in this case the (vector) P\'eclet number is $(-2.67878,\,6.46716)$.

Figure \ref{fig2} shows the values of the functional $J$ for 2D advection-diffusion equation, with varying velocities $\a=(\cos\alpha, \sin\alpha)$. We may observe that for fixed velocity $J$ is convex
\begin{figure}[h!]
\begin{center}
\begin{tabular}{ll}(a)  Functional $J$ for 1D problem. &(b)  Functional $J$ for 2D problem. \\
\includegraphics[width=0.5\linewidth]{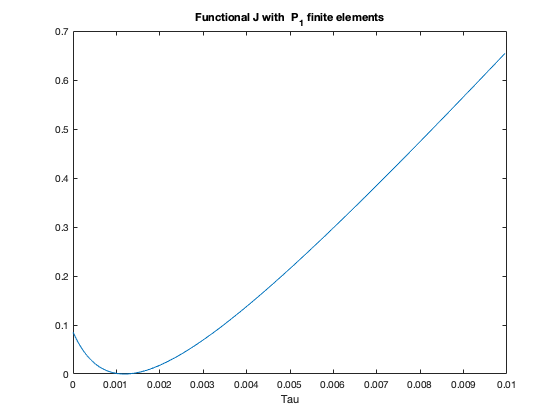}&\includegraphics[width=0.5\linewidth]{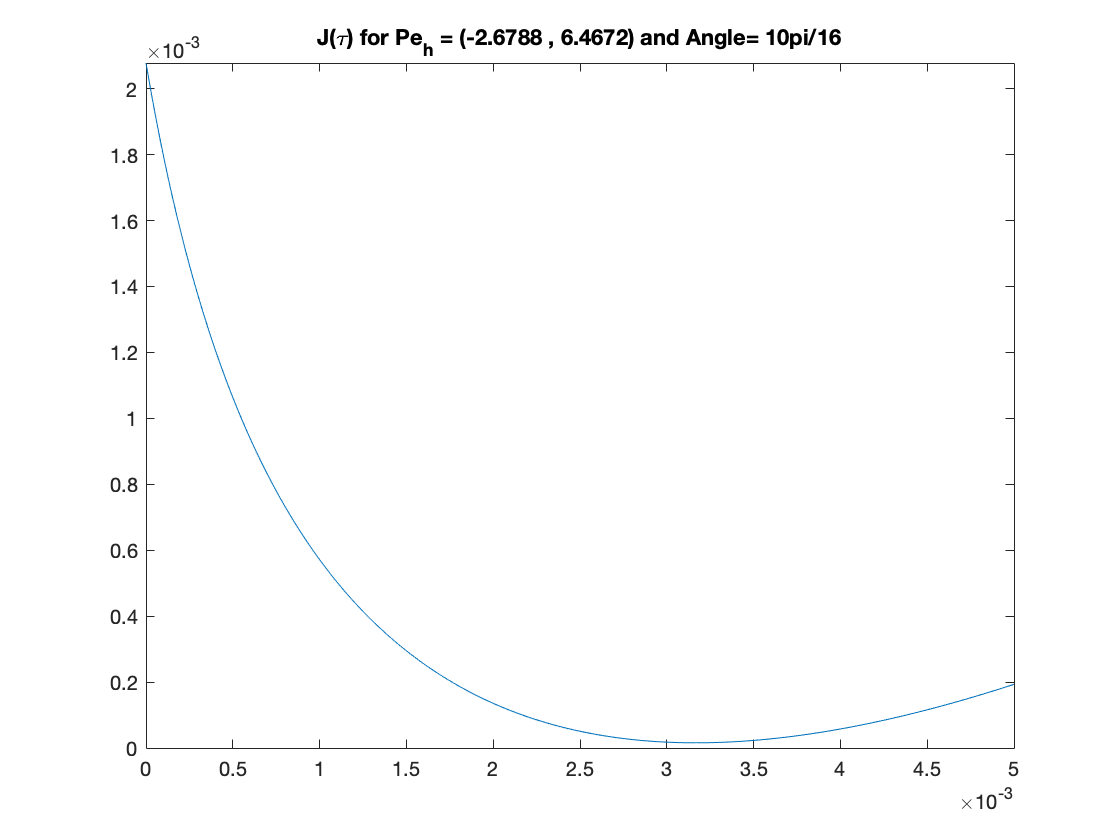} \\
\end{tabular}
\end{center}
\caption{\label{fig1} {Representation of the functional J for the stabilised solution of the advection-diffusion equations in 1D (panel (a)), corresponding to $Pe=1.6667$ and 2D (panel (b)) with $\P_1$ finite elements, corresponding to $Pe=(-2.67878,\,6.46716)$. }}
\end{figure}

\begin{figure}[h!]
\begin{center}
\includegraphics[width=0.7\linewidth]{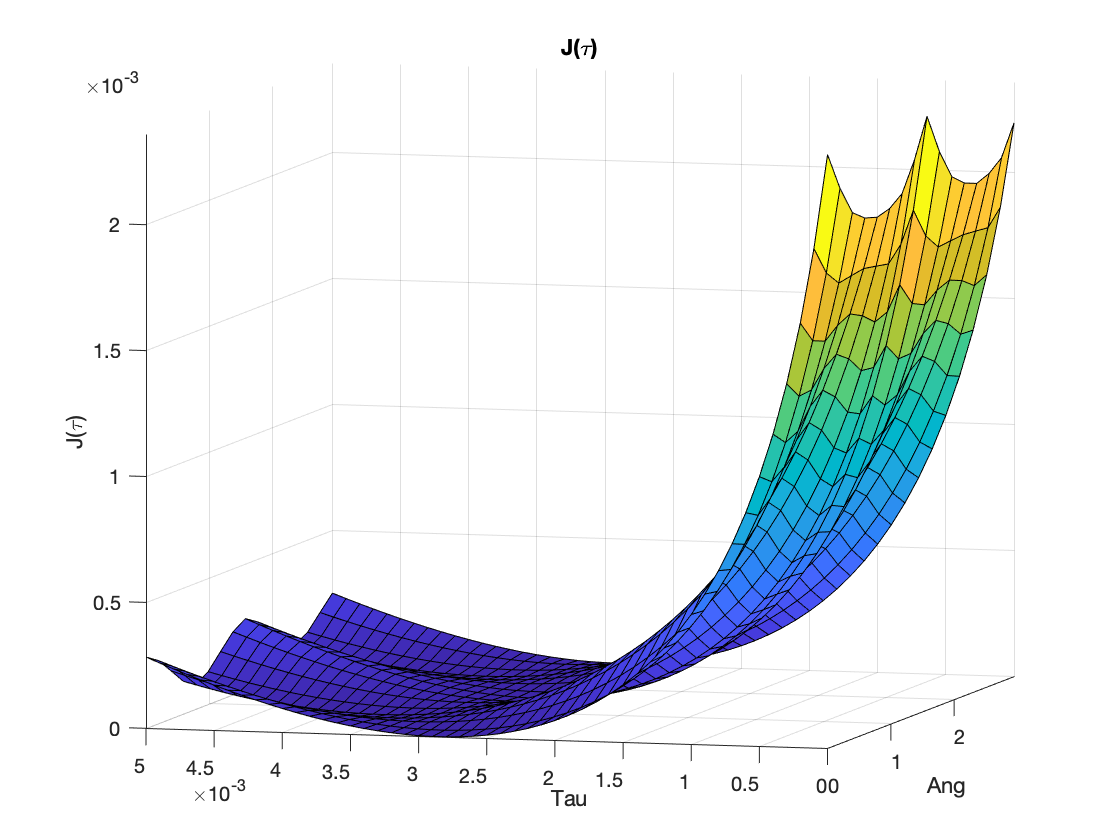} 
\end{center}
\caption{\label{fig2} {Representation of the functional J for the stabilised solution of the advection-diffusion equations in 2D with velocities $\a=(\cos\alpha, \sin\alpha)$. }}
\end{figure}

\section{Numerical tests}\label{se:numtest}
We afford in this section the numerical testing of the optimal stabilised coefficients computed as stated in Section \ref{se:least} for $\P_1$, $\P_2$ and $\P_3$ finite elements. We shall consider two different test problems: 2D linear advection-diffusion with constant and variable velocity, and 2D Navier-Stokes flow, in the latter we shall also use the $\P_1$+Bubble finite element space. 

For each test, we shall compare the results provided by using the least-squares stabilised coefficients, to those obtained through several formulae for these coefficients. In particular, {we will consider two classical isotropic and four anisotropic coefficients}. 

As classical coefficients, at first, we consider the generalization to 2D of the optimal stabilization coefficient in 1D \cite{christie,johnkno}, that is,
\begin{equation}\label{tau1D}
 \tau_{1D}^{(K)} =\displaystyle \frac{  \mu}{\|\mathbf{a}_K\|^2}\,\varphi(Pe_h),\,\mbox{ with  } \varphi(P)=(P \coth(P)-1),
\end{equation}
where $Pe_h=(h_K\|\mathbf{a}\|)/(2\mu)$ is the element P\'eclet number.

Second, the stabilisation coefficient through orthogonal sub-scales in finite element methods proposed by Codina in \cite{codina12} are, 
\begin{equation}\label{taucodina}
\tau^{(K)}_C=\left(\left(4\frac{\mu}{h_K^2}\right)^2+\left(2\frac{\|\mathbf{a}_K\|}{h_K}\right)^2\right)^{-1/2}.
\end{equation}

{As anisotropic coefficients, first, we consider the stabilization coefficient of the VMS-spectral method introduced by Chac\'on et al. in \cite{chafergom}. These coefficients are computed by spectral solution of the subgrid scales, and are not expressed through analytic formulas. }

{Next, we consider the anisotropic version of Codina coefficient proposed by Colom\'es et. al. in \cite{Colomes2018}, which consists in considering $h_{flow,K}$ instead of $h_K$ in the convection-dominated regime, that is, 
\begin{equation}\label{taucodinacolomes}
\tau^{(K)}_{CC}=\left(\left(4\frac{\mu}{h_K^2}\right)^2+\left(2\frac{\|\mathbf{a}_K\|}{h_{flow,K}}\right)^2\right)^{-1/2}.
\end{equation}}

{Moreover, we consider the stabilization parameter based on $L_2$ norm which was proposed by Hauke et. al. in \cite{Hauke2008},
\begin{equation}\label{tauhauke}
\tau^{(K)}_{H}=\min\left(\frac{h_{flow,K}}{\sqrt{3}\|\mathbf{a}_K\|},\frac{h_{K}^2}{24.24 \mu}\right).
\end{equation}}

{We  also consider the stabilization coefficient that was introduced by Franca and Valentin in \cite{Franca2000} and extended by Hauke et. al. in \cite{Hauke2002}, which in our case is given by,
\begin{equation}\label{tauflow}
\tau^{(K)}_{flow}=\left(\frac{2\mu}{m_K h_{K}^2}\xi(Pe)\right)^{-1},
\end{equation}
where $Pe=(m_k \|\mathbf{a}_K\|h_{K})/\mu$, $m_k=1/3$ and
\begin{equation}\label{defxi}
\xi(x)=\left\{\begin{array}{ll}
1& \mbox{if }0\leq x \leq 1, \\
x&\mbox{if }x>1.
\end{array}\right.
\end{equation}}

We will finally consider an anisotropic version of our least-squares stabilised coefficient, specifically
\begin{equation} \label{taufi}
\tau_{LSflow}^{(K)}=\displaystyle \frac{h_{flow,K}}{\|a\|_K}\,\varphi(P).
\end{equation}
where $h_{flow,K}$ is the elemental length along the flow direction for the advective terms.  

These stabilised coefficients are initially designed for $\P_1$ finite elements, to adapt them to $\P_k$ finite elements for integer $k \ge 2$, we replace $h_K$ by $h_K/k$. We do so because the distance between consecutive Lagrange interpolation points for $P_k$ on the sides or edges of element $K$ is $h_K/k$. Anyhow we only consider the coefficients obtained by spectral VMS technique for $\P_1$ finite elements, as these are not known for higher interpolation order. 

We divide the results into four subsections. In the first one, we focus on the advection-diffusion equation with constant velocity while in the second one, we consider the same problem with anisotropic velocities in a structured mesh. The third one is devoted to solve the advection-diffusion equation on unstructured meshes, in a flow around a cylinder. {We finally consider a tests to solve Navier-Stokes equations, actually the lid-driven cavity flow. }


\subsection{Advection-diffusion problems}
\subsubsection{Test 1: Advection-diffusion problem, constant advection velocity with varying orientation}
In this section, we compare the performances of all considered stabilised coefficients for the advection-diffusion problem (\ref{WEAD}) with constant velocity, but for several orientations and P\'eclet numbers.



We actually consider the unit square with homogeneous Dirichlet boundary conditions, 
with diffusion coefficient $\mu=1$ and source term $f(x,y)=\sin(\pi x)\cos(\pi y)$. We take triangular meshes of isosceles right triangles with sides of length $h=1/120,$ for $\P_1$, $h=1/60,$ for $\P_2$ and $h=1/40,$ for $\P_3$ finite elements, respectively. We consider constant velocities of the form  $\mathbf{a}= (k\sqrt{2}\cos{\alpha},k\sqrt{2}\sin{\alpha}),$ for $k=400,\,800,\,1600,\, \cdots,\,102400$, where $\alpha=n\pi/10,$ for $n=0,2,\ldots,18.$ In these cases, the global P\'eclet number ${Pe_h}$  varies between 9.4281 and 603.398 for $\P_1$ and $\P_2$ finite elements, respectively, and between 7.0711 and 905.097 for $\P_3$\coln.

{\coln In Figures \ref{errores1} and \ref{errores3}, we represent for each $n=0,2,\ldots,18,$ the mean errors in $L^2$ and $L^\infty$ norms, of the stabilised solution obtained through the least-squares and  stabilised coefficients and those stated in \eqref{tau1D}-\eqref{taufi},  for $\P_1$ and $\P_3$ f. e., respectively. The VMS-spectral stabilised solution are computed only $\P_1$ f. e.. The mean errors are averaged with respect to the P\'eclet numbers for the ranges of values of ${Pe_h}$ specified before and we compare the stabilised solution with a reference solution of the problem computed in a grid with grid size $h=1/1200$ for $\P_1$ f. e. and $h=1/400$ for $\P_3$ f. e.. \coln These errors roughly behave as periodic functions with period $\pi$ in all the cases. }
\begin{figure}[h!]
\begin{center}
\begin{tabular}{ll}(a)  Errors in norm $L^2$ classical coeffs. &(b)  Errors in norm $L^\infty$ classical coeffs. \\
\includegraphics[width=0.5\linewidth]{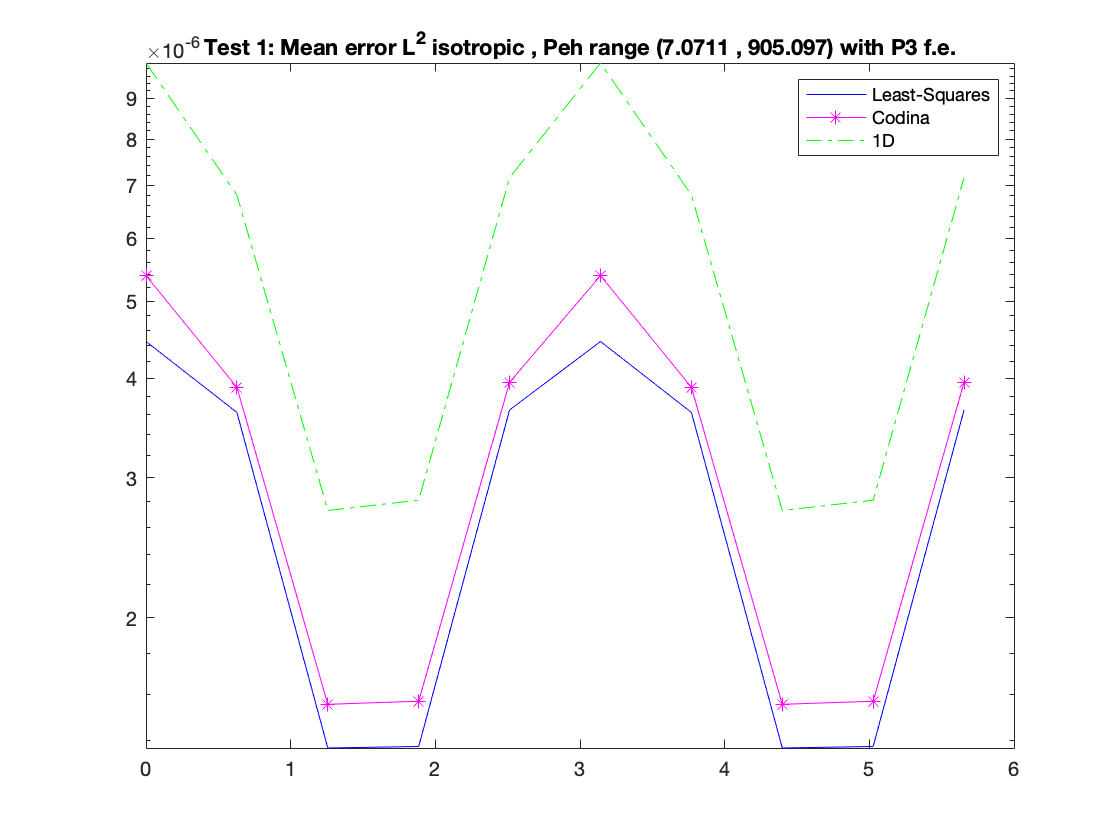}&\includegraphics[width=0.5\linewidth]{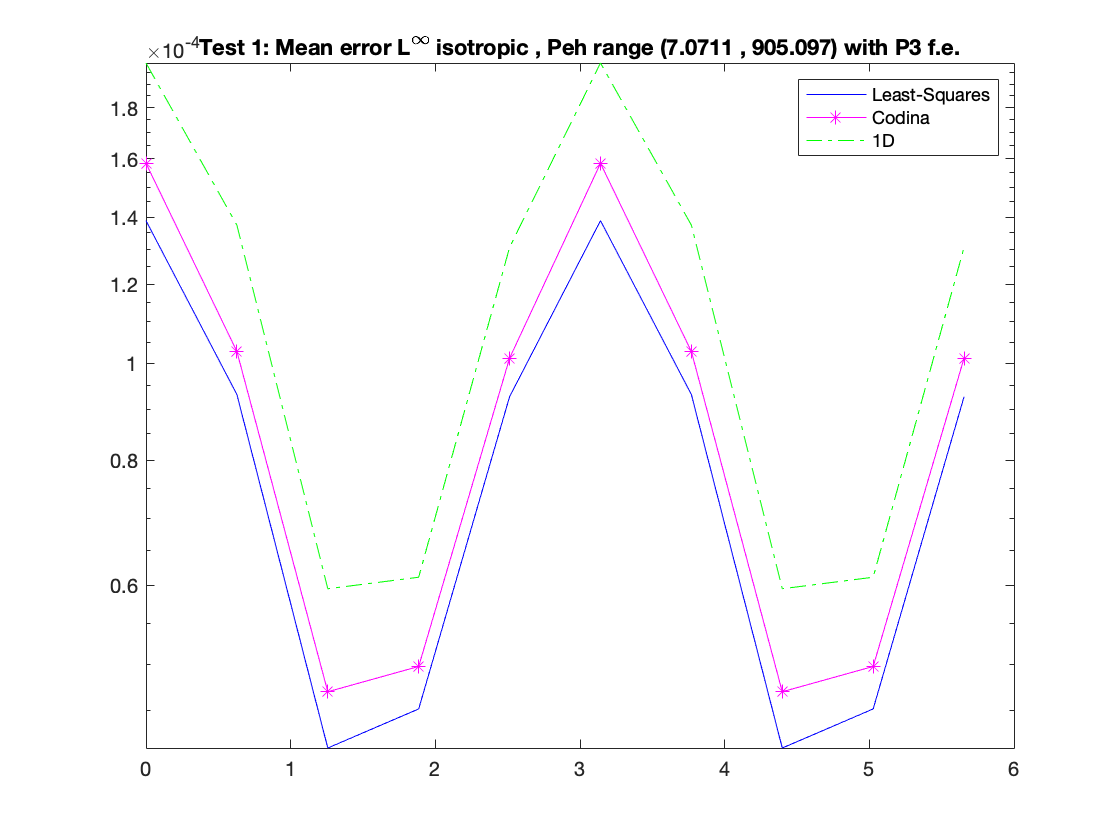} \\
{(c)  Errors in norm $L^2$ anisotropic coeffs.} &{(d)  Errors in norm $L^\infty$ anisotropic coeffs.}\\
\includegraphics[width=0.5\linewidth]{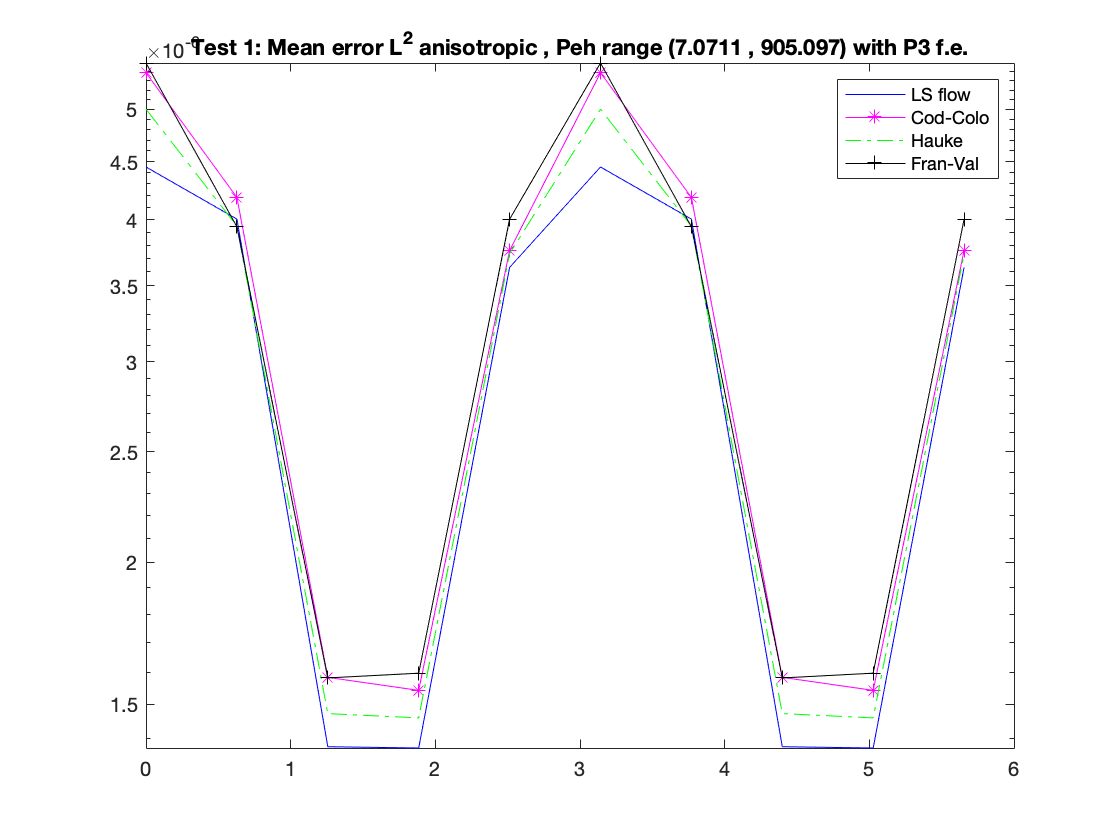}&\includegraphics[width=0.5\linewidth]{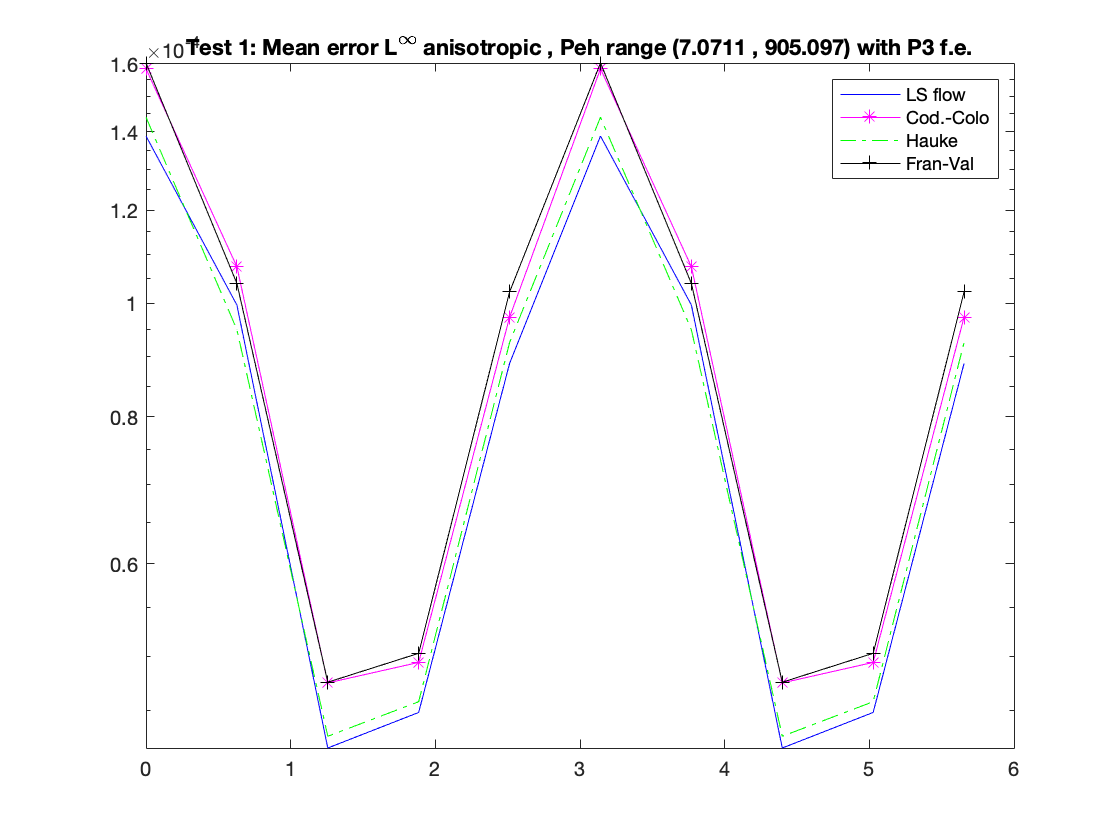}\\
\end{tabular}
\end{center}
\caption{\label{errores1} {Test 1: Representation of the errors obtained with $\mathbb{P}_1$ finite elements and different schemes in norms $L^2$ (panels (a) and (c)) and $L^\infty$, (panels (b) and (d)) for $\alpha=n\pi/10,$ with $n=0,2,\ldots,18.$ In panels (a) and (b), we represent the errors obtained with the least-squares stabilised coefficients, the errors obtained with the VMS-spectral stabilised coefficients and the errors obtained with the classical coefficients. In panels (c) and (d), we represent the errors obtained with all the anisotropic coefficients, in particular the least-squares ones.}}
\end{figure}
\begin{figure}[h!]
\begin{center}
\begin{tabular}{ll}(a)  Errors in norm $L^2$ classical coeffs. &(b)  Errors in norm $L^\infty$ classical coeffs. \\
\includegraphics[width=0.5\linewidth]{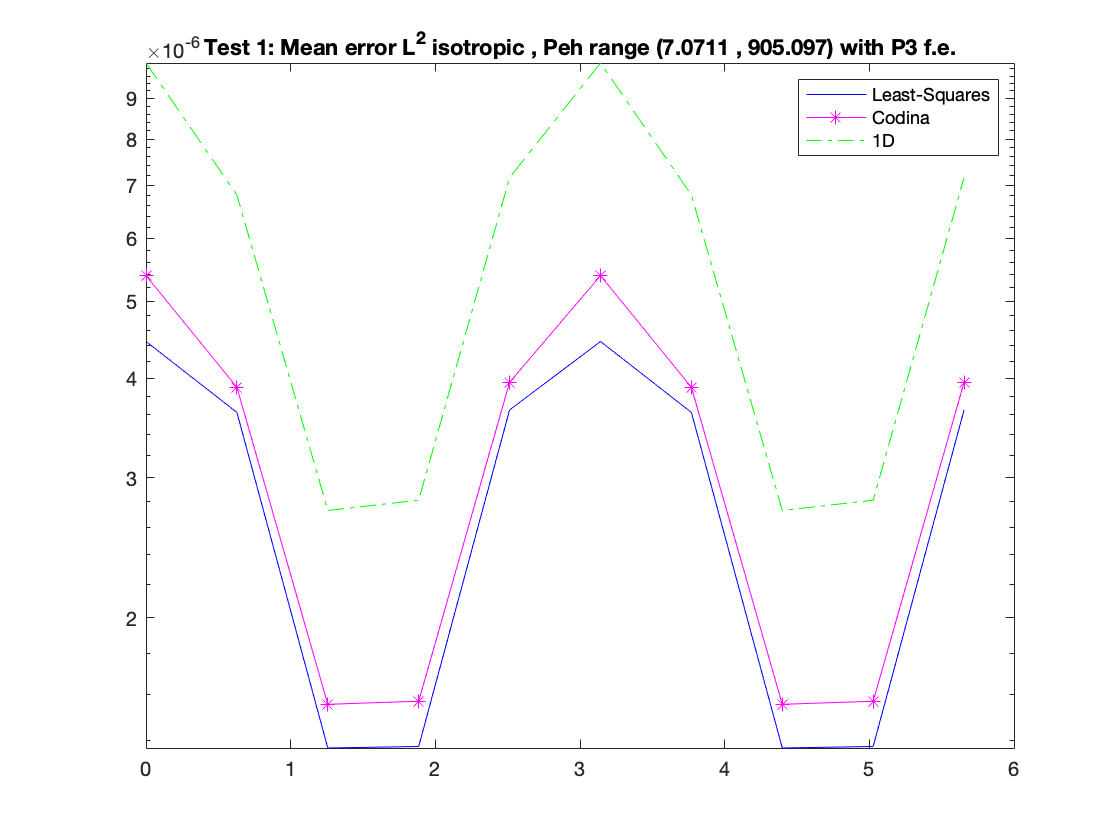}&\includegraphics[width=0.5\linewidth]{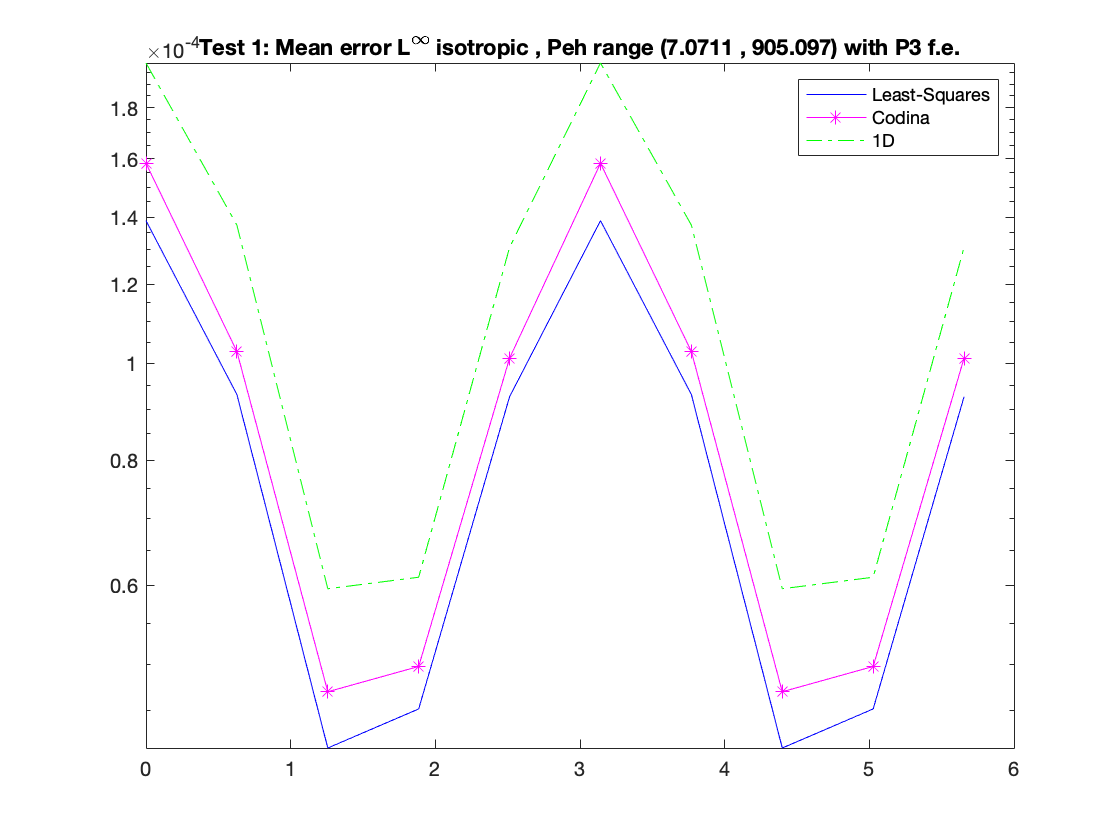} \\
{(c)  Errors in norm $L^2$ anisotropic coeffs.} &{(d)  Errors in norm $L^\infty$ anisotropic coeffs.}\\
\includegraphics[width=0.5\linewidth]{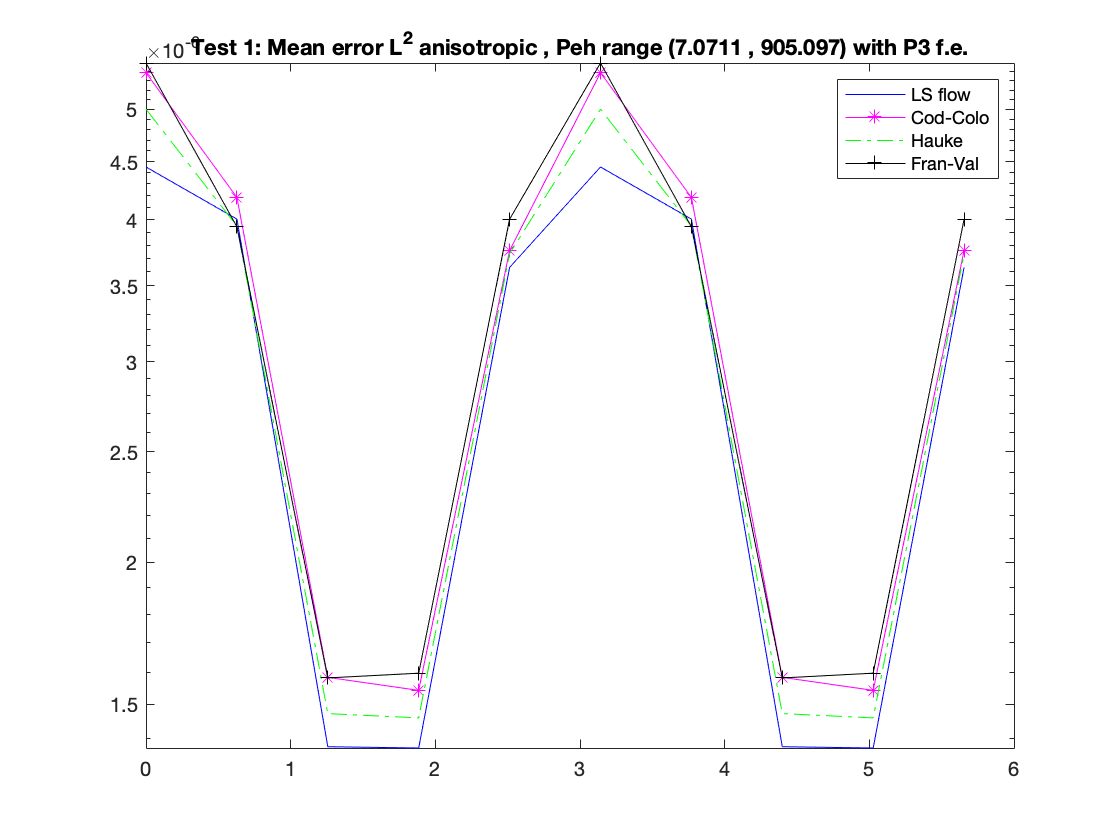}&\includegraphics[width=0.5\linewidth]{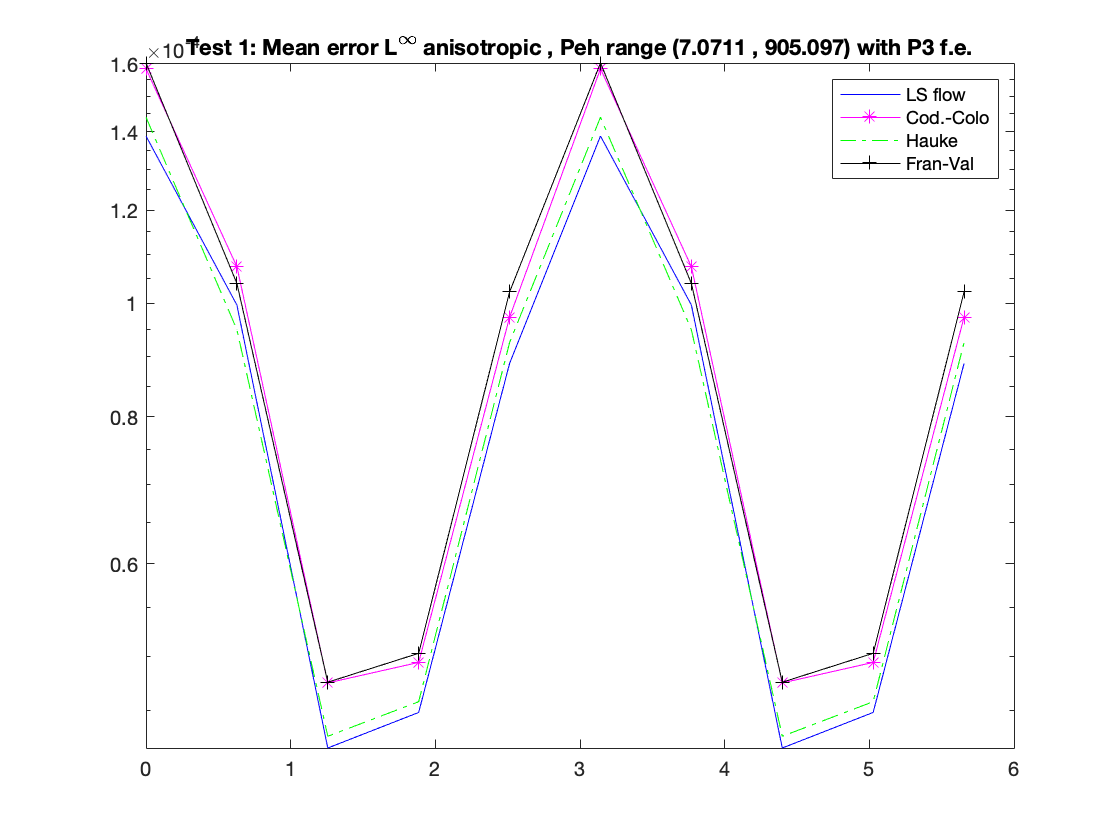}\\
\end{tabular}
\end{center}
\caption{\label{errores3} {Test 1: Representation of the errors obtained with $\mathbb{P}_3$ finite elements and different schemes in norms $L^2$ (panels (a) and (c)) and $L^\infty$, (panels (b) and (d)) for $\alpha=n\pi/10,$ with $n=0,2,\ldots,18.$ In panels (a) and (b), we represent the errors obtained with the least-squares stabilised coefficients, the errors obtained with the VMS-spectral stabilised coefficients and the errors obtained with the classical coefficients. In panels (c) and (d), we represent the errors obtained with all the anisotropic coefficients, in particular the least-squares ones.}}
\end{figure}

%
{\coln In Table \ref{tab:table1}, we provide, for $\P_1$, $\P_2$ and $\P_3$ finite elements, the mean errors in $L^2$ and $L^\infty$ norms, averaged with respect to both the angles $\alpha=n\pi/10,$ for $n=0,2,\ldots,18$ and the P\'eclet numbers ranging between \coln 9.4281 and 603.398 for $\P_1$ and $\P_2$ f. e., and between 7.0711 and 905.097 for $\P_3$, respectively. We observe that the least squares stabilised coefficients provide the smallest errors for three types of f. e. in both $L^2$ and $L^\infty$ norms. We also observe that using the anisotropic version of the least-squares stabilised coefficients does not improve the error in any case, although these coefficients provide the second best errors. This is likely due to the fitting of the least-squares procedure to compute the stabilised coefficients to the actual geometrical element size. 
}
 \begin{table}[h!]
  \begin{center}
    { \begin{tabular}{l c l c }  
   \hline\hline                        
   $\P_1$, &${Pe_h} $-range& (9.4281, &603.398),   \\   \hline                   
   $L^2_{LS}$ &\textbf{1.8547e-06}&$L^\infty_{LS}$ & \textbf{5.3824e-05}\\
   $L^2_{VMS}$ &1.9564e-06&$L^\infty_{VMS}$&5.4445e-05  \\  
   $L^2_{1D}$ &1.8965e-06&$L^\infty_{1D}$ & 5.4329e-05\\ 
   $L^2_C$ &1.9223e-06&$L^\infty_C$ & 5.4521e-05\\ 
   $L^2_{LSflow}$ &1.8547e-06&$L^\infty_{LSflow}$ & 5.3908e-05\\
   $L^2_{CC}$ &1.9322e-06&$L^\infty_{CC}$ & 5.4551e-05\\ 
   $L^2_H$ &1.9948e-06&$L^\infty_H$ & 5.4912e-05\\ 
   $L^2_{flow}$ &1.9237e-06&$L^\infty_{flow}$ & 5.4530e-05\\
   [0.5ex]        
   \hline\hline                        
 $\P_2$, &${Pe_h} $-range& (9.4281, &603.398),    \\   \hline                   
   $L^2_{LS}$ &\textbf{2.6855e-06}&$L^\infty_{LS}$ &  \textbf{6.6932e-05}\\
   $L^2_{1D}$ &3.6969e-06&$L^\infty_{1D}$ & 7.7575e-05\\ 
   $L^2_C$ &2.8943e-06&$L^\infty_C$ & 7.1192e-05\\ 
   $L^2_{LSflow}$ &2.7045e-06&$L^\infty_{LSflow}$ & 6.7044e-05\\
   $L^2_{CC}$ &2.9102e-06&$L^\infty_{CC}$ & 7.1380e-05\\ 
   $L^2_H$ &2.8930e-06&$L^\infty_H$ & 7.0498e-05\\ 
   $L^2_{flow}$ &2.9045e-06&$L^\infty_{flow}$ & 7.1329e-05\\
  [0.5ex]       
   \hline\hline                        
  $\P_3$, &${Pe_h} $-range& (7.0711, &905.097),    \\   \hline                   
   $L^2_{LS}$ &\textbf{2.8944e-06}&$L^\infty_{LS}$ &  \textbf{8.2187e-05}\\
   $L^2_{1D}$ &  5.8929e-06&$L^\infty_{1D}$ & 1.1763e-04\\ 
   $L^2_C$ &3.2725e-06&$L^\infty_C$ & 9.1822e-05\\ 
   $L^2_{LSflow}$ &2.9673e-06&$L^\infty_{LSflow}$ & 8.2774e-05\\
   $L^2_{CC}$ &3.2908e-06&$L^\infty_{CC}$ & 9.1999e-05\\ 
   $L^2_H$ &3.1234e-06&$L^\infty_H$ & 8.4009e-05\\ 
   $L^2_{flow}$ &3.3238e-06&$L^\infty_{flow}$ & 9.2847e-05\\
  [0.5ex]        
   \hline     
   \end{tabular}}
    \end{center}
 \caption{  \label{tab:table1} {Test 1. Averaged errors w. r. t. the angles $\alpha=n\pi/10,$ for $n=0,2,\ldots,18$ and the P\'eclet numbers in $L^2$ and $L^\infty$ norms for $\P_1$, $\P_2$ and $\P_3$ finite elements. The P\'eclet numbers range between 9.4281 and 603.398 for $\P_1$ and $\P_2$ f.e. and between 7.0711 and 905.097 for $\P_3$ f.e., respectively.
}}
    \end{table}   

\subsubsection{Test 2: Advection-diffusion equations with anisotropic velocities}
In this section we afford the testing of the least-squares coefficients for the advection-diffusion equations (\ref{WEAD}) with anisotropic velocities. We consider as domain the rectangle $\Omega=(0,1)\times(0,1/2)$ with homogeneous Dirichlet boundary conditions and source term $f(x,y)=1$. {\coln We consider diffusion coefficients $\mu$ that vary with values $\mu=1.25e-05, \,2.5e-05,$ $\,5.0e-05,\cdots,0.0016$ and advection velocity}

\begin{equation}\label{vel}
\mathbf{a}(x,y)=(a_1(x,y),a_2(x,y)),
\end{equation}
where
\begin{equation}\label{a1}
a_1(x,y)=\left\{\begin{array}{ll}
-0.1(y-0.5)&\mbox{if }\sqrt{(x-0.5)^2+(y-0.5)^2}<0.01, \\ \noalign{\smallskip}
-2(y-0.5)&\mbox{if }\sqrt{(x-0.5)^2+(y-0.5)^2}\geq 0.01,
\end{array}\right.
\end{equation}
and
\begin{equation}\label{a2}
a_2(x,y)=\left\{\begin{array}{ll}
0.1(x-0.5)&\mbox{if }\sqrt{(x-0.5)^2+(y-0.5)^2}<0.01, \\ \noalign{\smallskip}
2(x-0.5)&\mbox{if }\sqrt{(x-0.5)^2+(y-0.5)^2}\geq 0.01.
\end{array}\right.
\end{equation}
In Figure \ref{fig_vel} we represent the velocity vector field $\mathbf{a}$ given in (\ref{vel})-(\ref{a2}). 
\begin{figure}[h!]
\begin{center}
\includegraphics[width=0.6\linewidth]{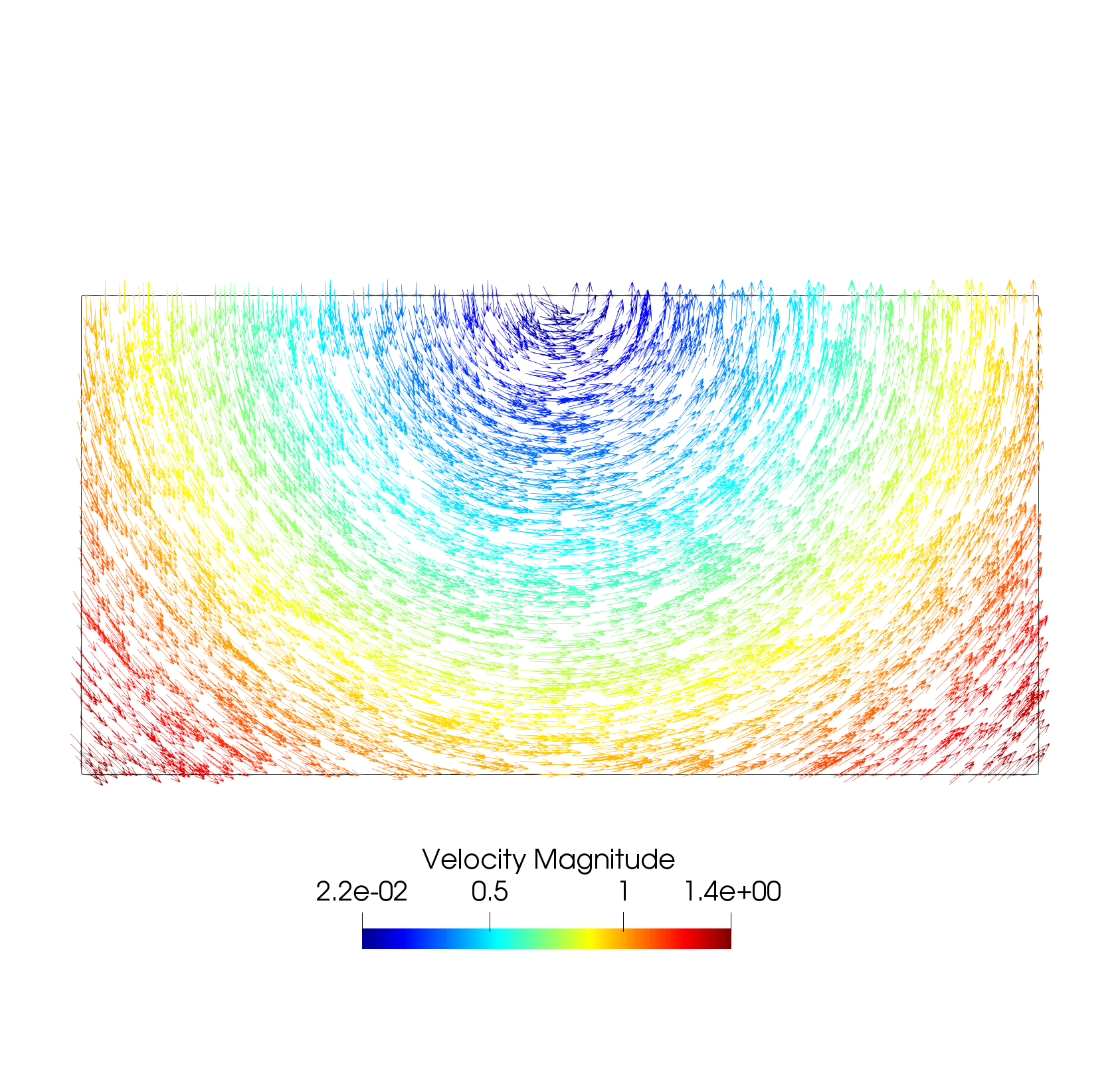}
\end{center}
\caption{\label{fig_vel} Representation of the velocity $\mathbf{a}(x,y)$.}
\end{figure}

 {\coln In Table \ref{tab:table2} we provide, for $\P_1$, $\P_2$ and $\P_3$ finite elements on a grid of size $h=1/96$, the mean errors in $L^2$ and $L^\infty$ norms, averaged with respect to  the P\'eclet numbers. These vary within the range between 5.1201 and 655.378. The errors have been computed by comparing the stabilised solution in a mesh of $N=96$ with  a  reference solution of (\ref{WEAD}) obtained with a refined mesh of $N^*=12N$ for $\P_1$ f. e., with a refined mesh of $N^*=6N$ for $\P_2$ f. e. and with a refined mesh of $N^*=4N$ for $\P_3$ f. e.. 
In this test we compute the stabilised solution through the least-squares stabilised coefficients, through the anisotropic least-squares stabilised coefficients \eqref{taufi},  through the VMS-spectral stabilised coefficients (only $\P_1$ f. e.) and through  the stabilised coefficients given in \eqref{taucodina}, \eqref{tauhauke} and \eqref{tauflow}.
 \begin{table}[h!]
 \begin{center}
     \begin{tabular}{l c l c }  
   \hline\hline                        
   $\P_1$, &${Pe_h} $-range &(5.1201, &655.378)  \\   \hline                   
   $L^2_{LS}$ &\textbf{0.044698}&$L^\infty_{LS}$ & \textbf{1.758}\\
   $L^2_{VMS}$ & 0.049077&$L^\infty_{VMS}$&1.7723  \\  
   $L^2_C$ &0.055891&$L^\infty_C$ & 1.8211\\ 
   $L^2_{LSflow}$ & 0.044724&$L^\infty_{LSflow}$ & 1.758\\
   $L^2_H$ &0.055089&$L^\infty_H$ & 1.8101\\ 
   $L^2_{flow}$ &0.045405 &$L^\infty_{flow}$ & 1.809\\
   [0.5ex]        

   \hline\hline                        
 $\P_2$, &${Pe_h} $-range & (5.1201, &655.378)  \\   \hline                   
   $L^2_{LS}$ &0.02198&$L^\infty_{LS}$ &   1.2736\\
   $L^2_C$ &0.029231&$L^\infty_C$ & 1.4014\\ 
   $L^2_{LSflow}$ &\textbf{0.021967}&$L^\infty_{LSflow}$ &  \textbf{1.2736}\\
   $L^2_H$ &0.029577&$L^\infty_H$ & 1.3835\\ 
   $L^2_{flow}$ &0.024069&$L^\infty_{flow}$ & 1.3925\\
  [0.5ex]       
   \hline\hline                        
  $\P_3$, &${Pe_h} $-range &(5.1201, &655.378)  \\   \hline                   
   $L^2_{LS}$ &0.014192&$L^\infty_{LS}$ & 1.3339\\
   $L^2_C$ &0.019374&$L^\infty_C$ & 1.5159\\ 
   $L^2_{LSflow}$ &\textbf{0.014183}&$L^\infty_{LSflow}$ & \textbf{1.3339}\\
   $L^2_H$ &0.019958&$L^\infty_H$ & 1.5224\\ 
   $L^2_{flow}$ &0.015532&$L^\infty_{flow}$ & 1.5034\\
  [0.5ex]        
 \hline     
   \end{tabular}
    \end{center}
      \caption{  \label{tab:table2}  {Test 2. Averaged errors w. r. t. the P\'eclet numbers in $L^2$ and $L^\infty$ norms for $\P_1$, $\P_2$ and $\P_3$ finite elements. The P\'eclet numbers with a maximum range between 5.1201 and 655.378 for three types of f. e.}}
\end{table}

 \begin{figure}[h!]
\begin{center}
\includegraphics[width=0.6\linewidth]{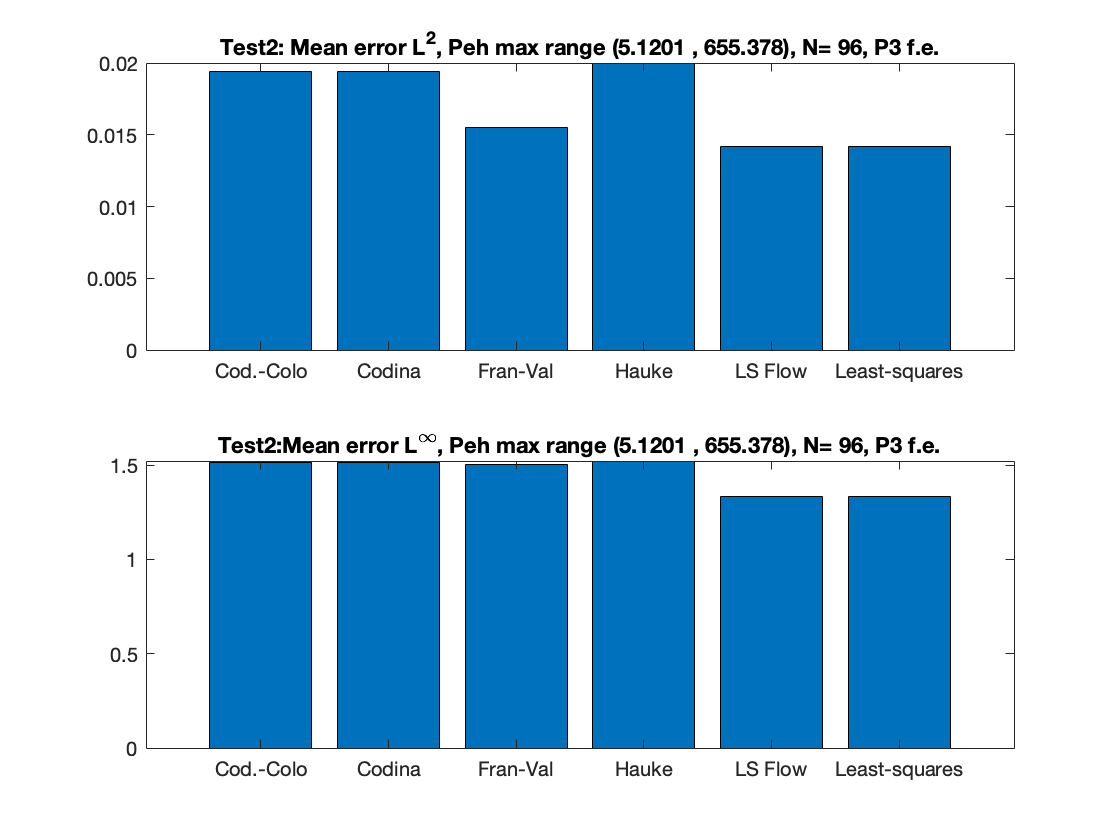}
\end{center}
\caption{\label{fig_errP3} Test 2. Errors for $\P_3$ finite elements with $N^*=4N$.}
\end{figure}

 {\color{black} In Fig. \ref{fig_errP3}, we represent  the mean errors in $L^2$ and $L^\infty$ norms for $\P_3$ f. e. for the same stabilized coefficients and data as in Table \ref{tab:table2}.  In Fig. \ref{fig_taus}, we represent in the case $N=96$ and $\mu=0.001$, the least-squares stabilized coefficients, panel (a) and the VMS-spectral stabilised coefficients, panel (b), for $\P_1$ f.e.. Finally, in Fig. \ref{fig_sol}, we represent, also for $N=96$ and $\mu=0.001$ the solution obtained by the least-squares stabilised coefficients for $\P_1$ f.e..     

\begin{figure}[ht!]
\begin{center}
\begin{tabular}{ll}(a)  Least-squares stabilised coefficients &(b) VMS-spectral stabilised coefficients\\
\includegraphics[width=0.5\linewidth]{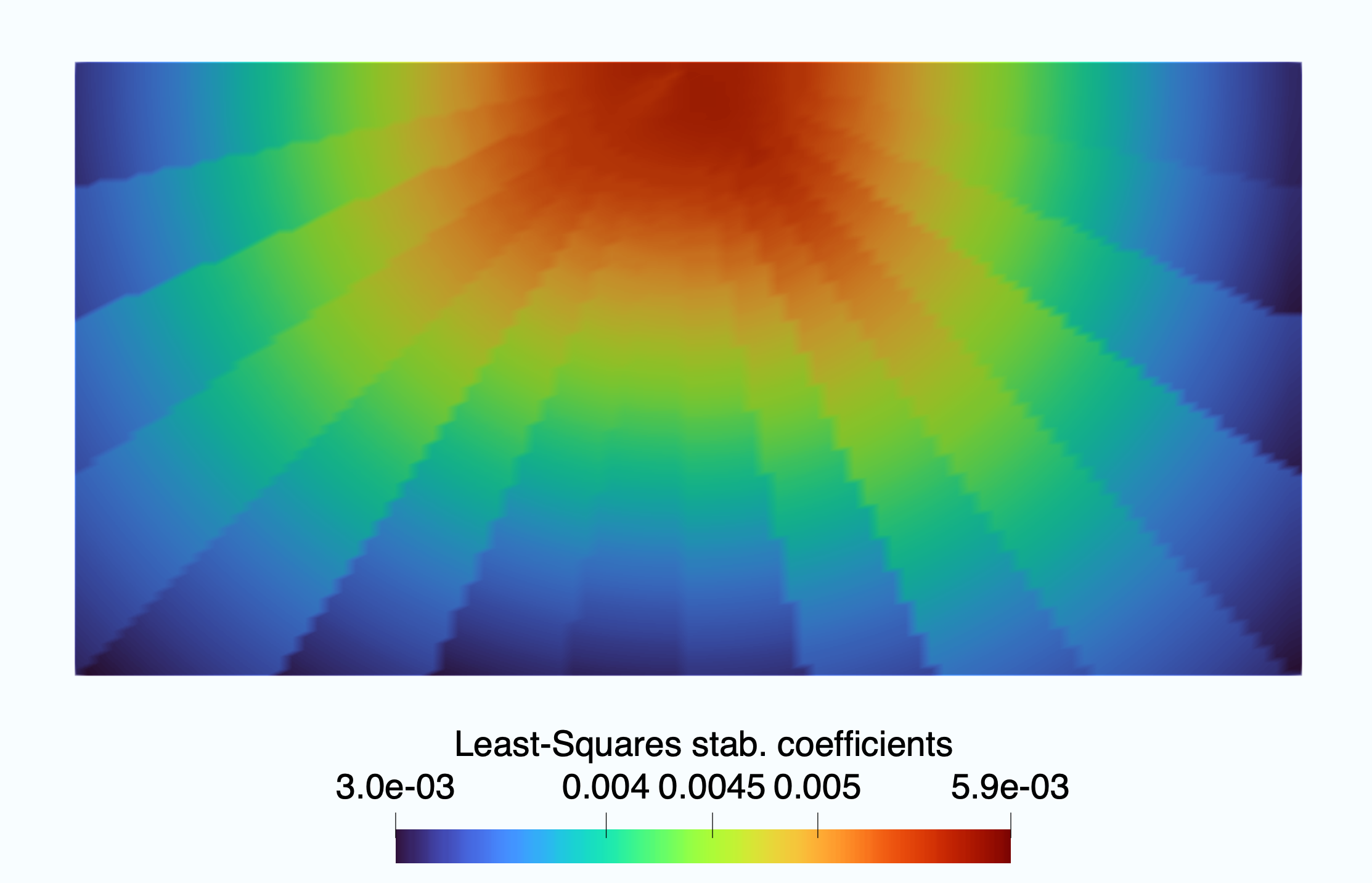}&\includegraphics[width=0.5\linewidth]{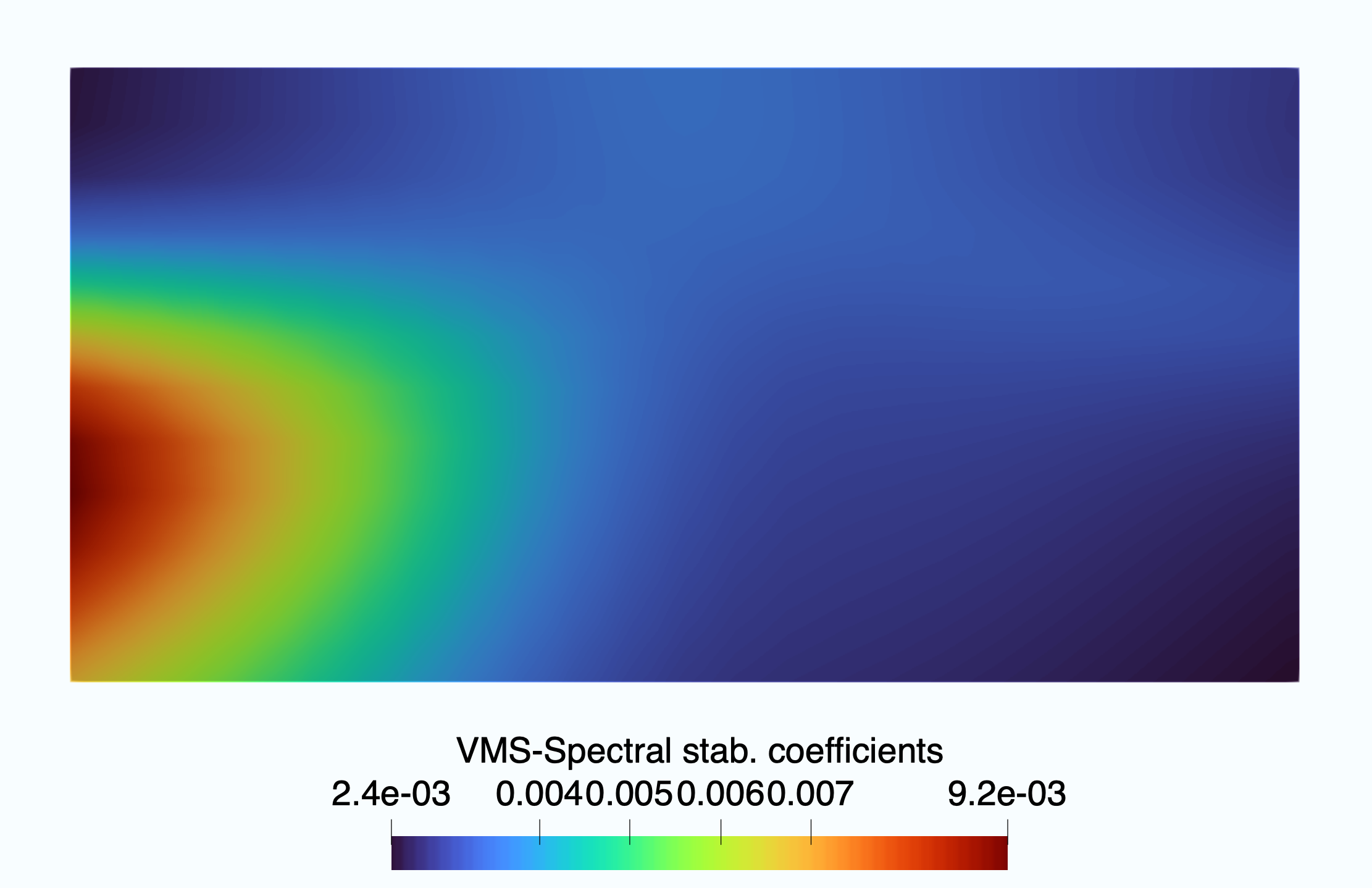}
\end{tabular}
\end{center}
\caption{\label{fig_taus} Test 2. Representation of the least-squares stabilised coefficients, (panel (a)) and the VMS-spectral stabilised coefficients, (panel (b)).}
\end{figure}
\begin{figure}[ht!]
\begin{center}
\includegraphics[width=0.7\linewidth]{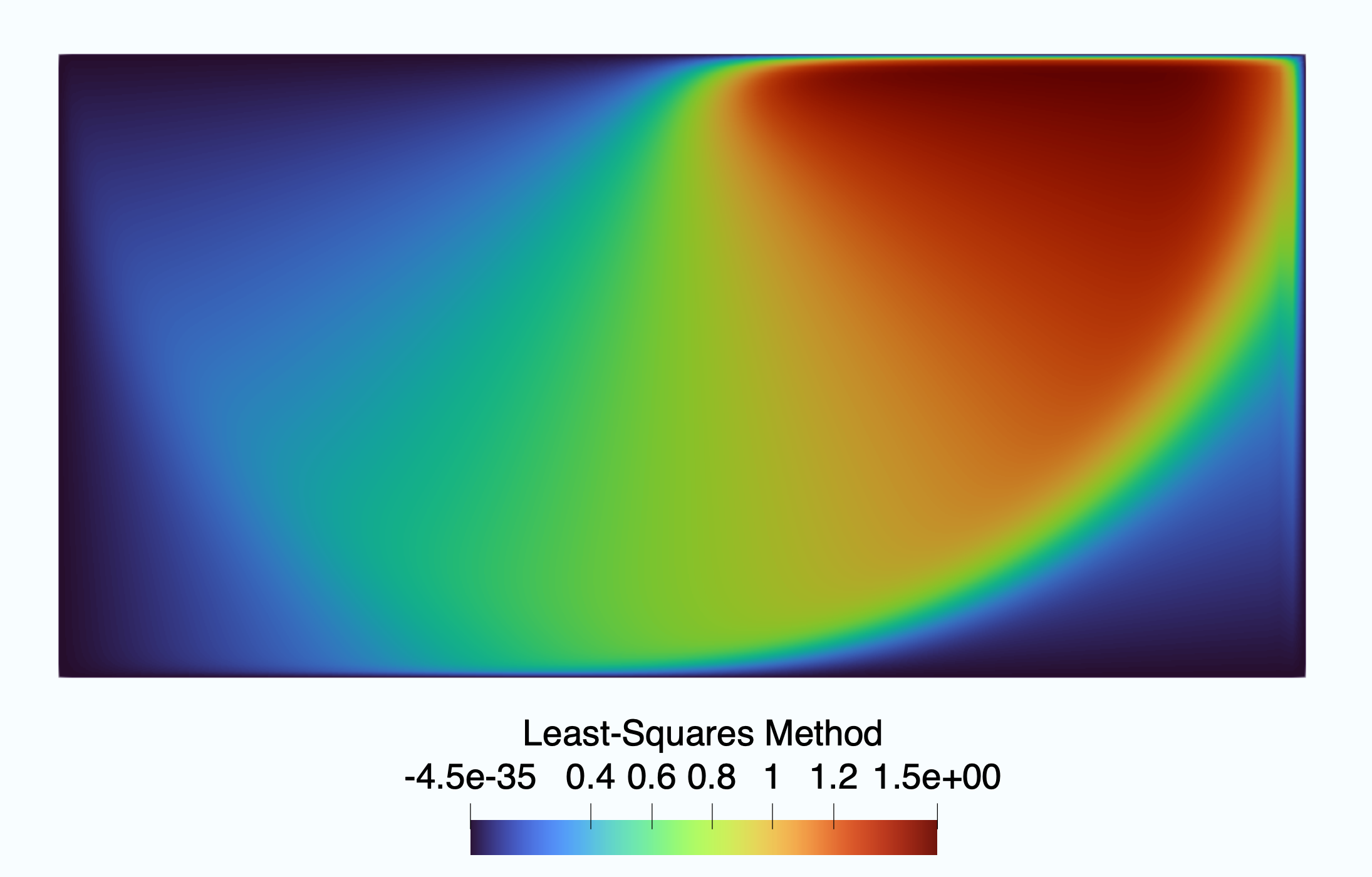} \\
\end{center}
\caption{\label{fig_sol} Test 2. Representation of the solution obtained with the least-squares stabilised coefficients. }
\end{figure}

 \subsubsection{Test 3: Advection-diffusion flow around a cylinder}
  The numerical tests considered until now have been carried on on structured meshes, for which the least-squares stabilised coefficients have been computed.  In this subsection, we check the reliability of the method in a more general case, with unstructured meshes.
  
In this test, we previously compute the steady state of a fluid with Reynolds (Re) number 100 around a cylinder. We use this velocity as the advection velocity $\mathbf{a}(x,y)$ to solve the advection-diffusion problem  (\ref{WEAD})  for a passive scalar.  
In Fig. \ref{fig_vel_cil}, we represent the velocity vector field $\mathbf{a}$.  We consider diffusion coefficients $\mu$ that vary with values $\mu=5.0e-06, \,7.5e-06,$ $\,1.0e-05,\cdots,0.0005$ and source term $f(x,y)=0$. 
\begin{figure}[ht!]
\begin{center}
\includegraphics[width=0.7\linewidth]{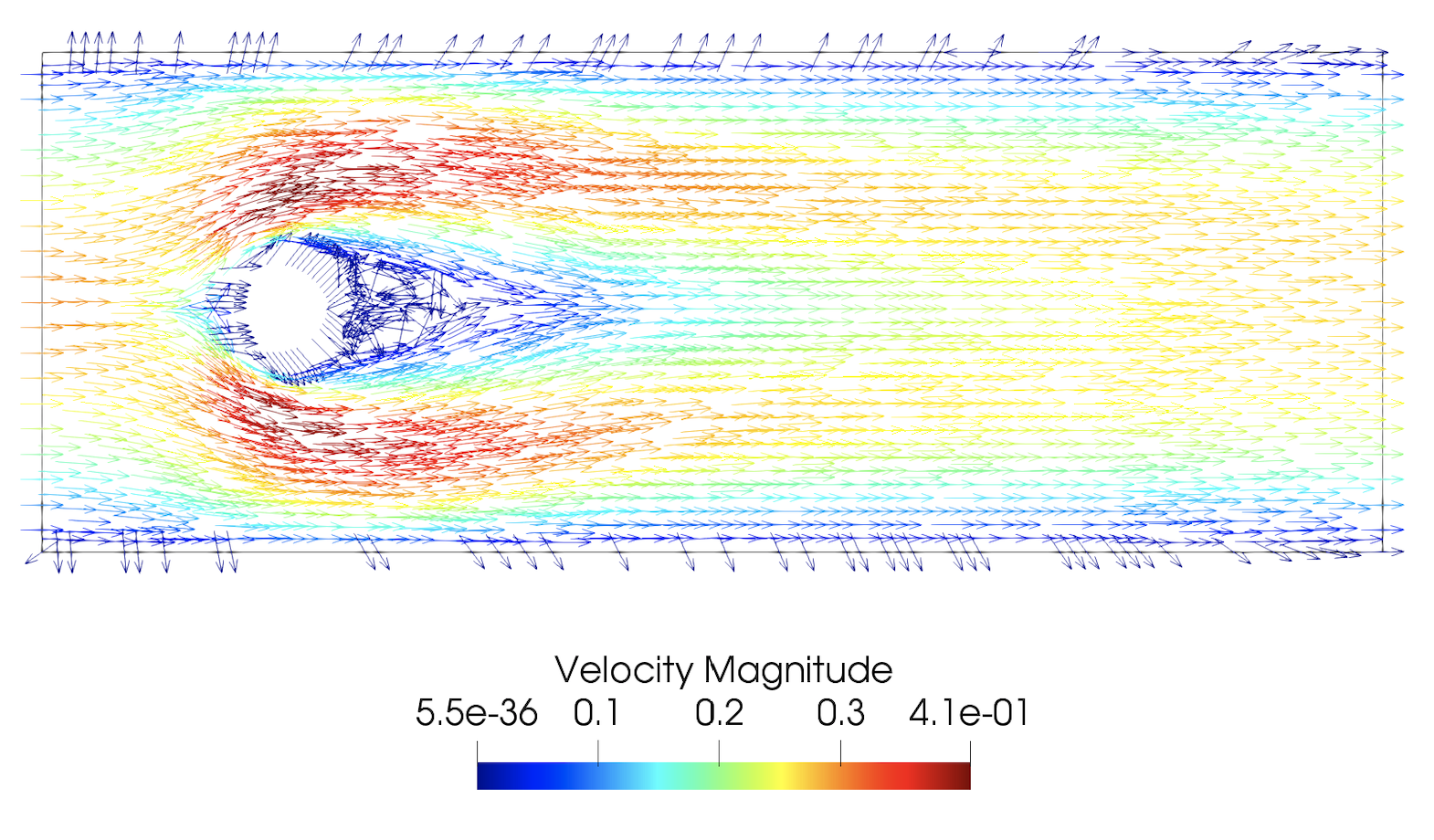}
\end{center}
\caption{\label{fig_vel_cil} Test 3. Representation of the velocity $\mathbf{a}(x,y)$.}
\end{figure}
 In Table \ref{tab:tablecilindro} we provide, for $\P_1$, $\P_2$ and $\P_3$ f.e., the mean errors in $L^2$ and $L^\infty$ norms, averaged with respect to  the P\'eclet numbers within the range 8.18483 to 818.483.  The errors have been computed by comparing the stabilised solution in this unstructured mesh with a reference solution of (\ref{WEAD}) obtained splitting the unstructured mesh by 4, for $\P_1$, $\P_2$ and $\P_3$ f. e..
In this test we compute the stabilised solution through the least-squares stabilised coefficients, and through  the stabilised coefficients given in \eqref{taucodina}, \eqref{tauhauke} and \eqref{tauflow}. We also use through the VMS-spectral stabilised coefficients, but only for $\P_1$ f. e..
\begin{table}[ht!]
  \begin{center}
   \begin{tabular}{c c c c c c}  
   \hline  \hline                        
   $\P_1$ f. e. &VMS-spectral &Codina  & Hauke & Franca-Valentin  & Least-squares \\    [0.5ex]
   \hline                   
   $L^2$ & 5.61e-03 &5.93e-03 & 5.82e-03 & \textbf{5.32e-03}  &  5.39e-03\\  
   $L^\infty$ &2.46e-01 & 2.85e-01& 2.56e-01 & \textbf{2.02e-01} & 2.41e-01\\ 
   [0.5ex]       
   \hline  
   \end{tabular}   
\medskip

   \begin{tabular}{c c c c c c}  
   \hline  \hline                        
  $\P_2$ f. e.  &Codina  & Hauke & Franca-Valentin  & Least-squares \\    [0.5ex]
   \hline                   
 $ L^2$ & 9.43e-04 & 9.52e-04 & 9.08e-04 &  \textbf{8.36e-04}\\  
   $L^\infty$  & 3.45e-02& 2.88e-02 & \textbf{2.64e-02} & 2.75e-02\\ 
   [0.5ex]       
   \hline     
   \end{tabular}
\medskip

   \begin{tabular}{ c c c c c}  
   \hline  \hline                        
   $\P_3$ f. e. &Codina  & Hauke & Franca-Valentin  & Least-squares\\    [0.5ex]
   \hline                   
 $ L^2$ &  3.07e-04 & 3.11e-04 & 3.04e-04  &  \textbf{2.86e-04}\\  
   $L^\infty$ &  1.59e-02& 1.25e-02 & \textbf{1.21e-02} & 1.23e-02\\ 
   [0.5ex]       
   \hline     
   \end{tabular}
    \end{center}
    \caption{\label{tab:tablecilindro} Test 3. Errors  in $L^2$ and $L^\infty$ norms with several stabilisation coefficients.} 
    \end{table}
 
{\color{black} From the results in Table \ref{tab:tablecilindro}, let us note that the best performance is provided by least-squares stabilisation coefficients in some cases and by Franca and Valentin coefficients in other cases, although in the latter cases with values very close.}

{\color{black}In Fig. \ref{fig_taus_cil}, we map, for $\mu=0.0005$ and $\P_1$ f. e., the stabilisation coefficients obtained with the least-squares stabilised coefficients, (panel (a)), VMS-spectral stabilised coefficients (panel (b)), Hauke stabilised coefficients(panel (c)) and Franca and Valentin stabilised coefficients (panel (d)). Note that the patterns of the values reached by the coefficients cases are similar in all four cases (higher values of the stabilised coefficients near the inflow and the outflow boundaries), although the ranges of these values are quite different.} 
\begin{figure}[ht!]
\begin{center}
\begin{tabular}{ll}(a) Least-squares stabilised coefficients &(b) Franca-Valentin stabilised coefficients\\ 
\includegraphics[width=0.5\linewidth]{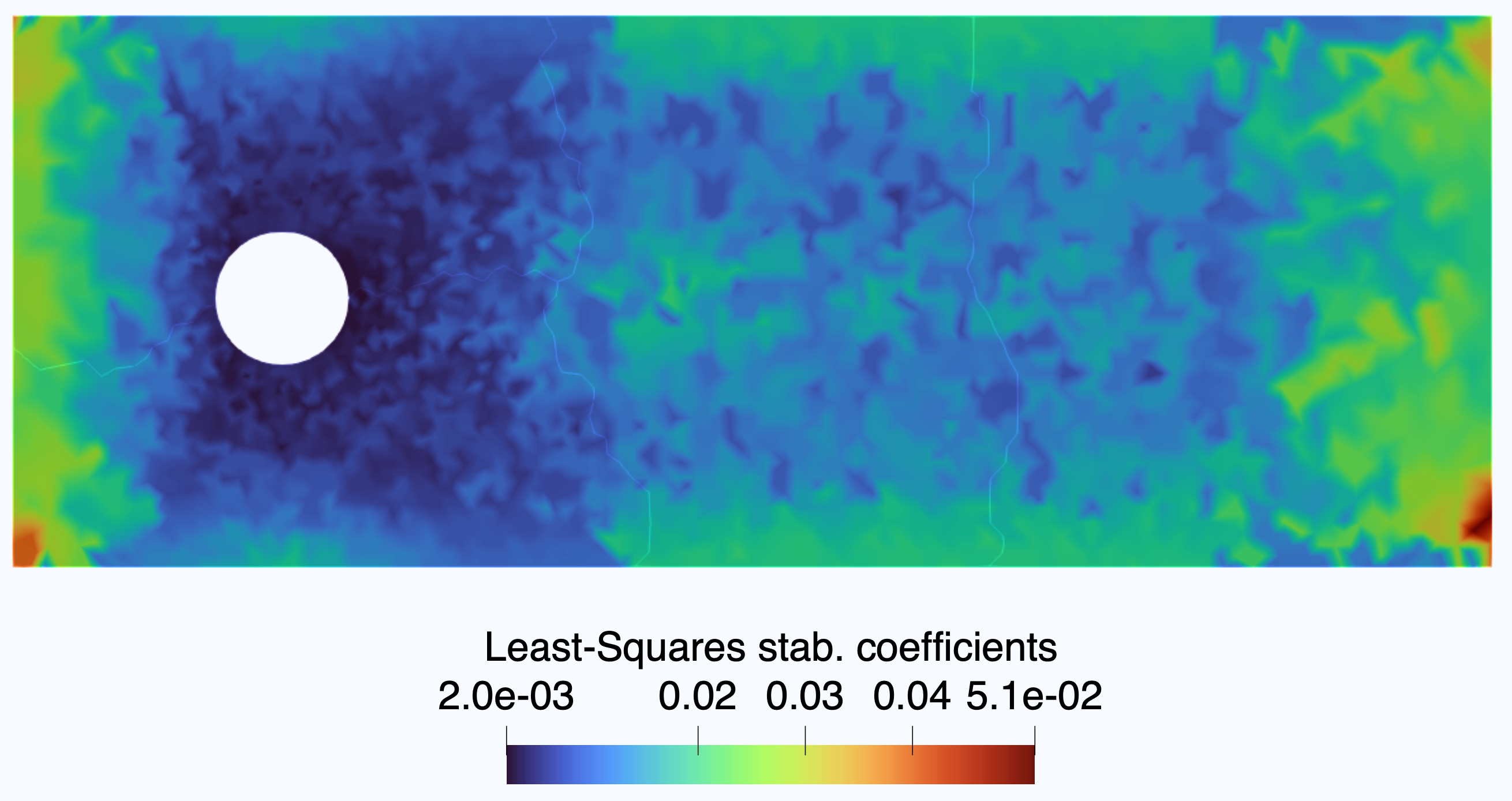}&\includegraphics[width=0.5\linewidth]{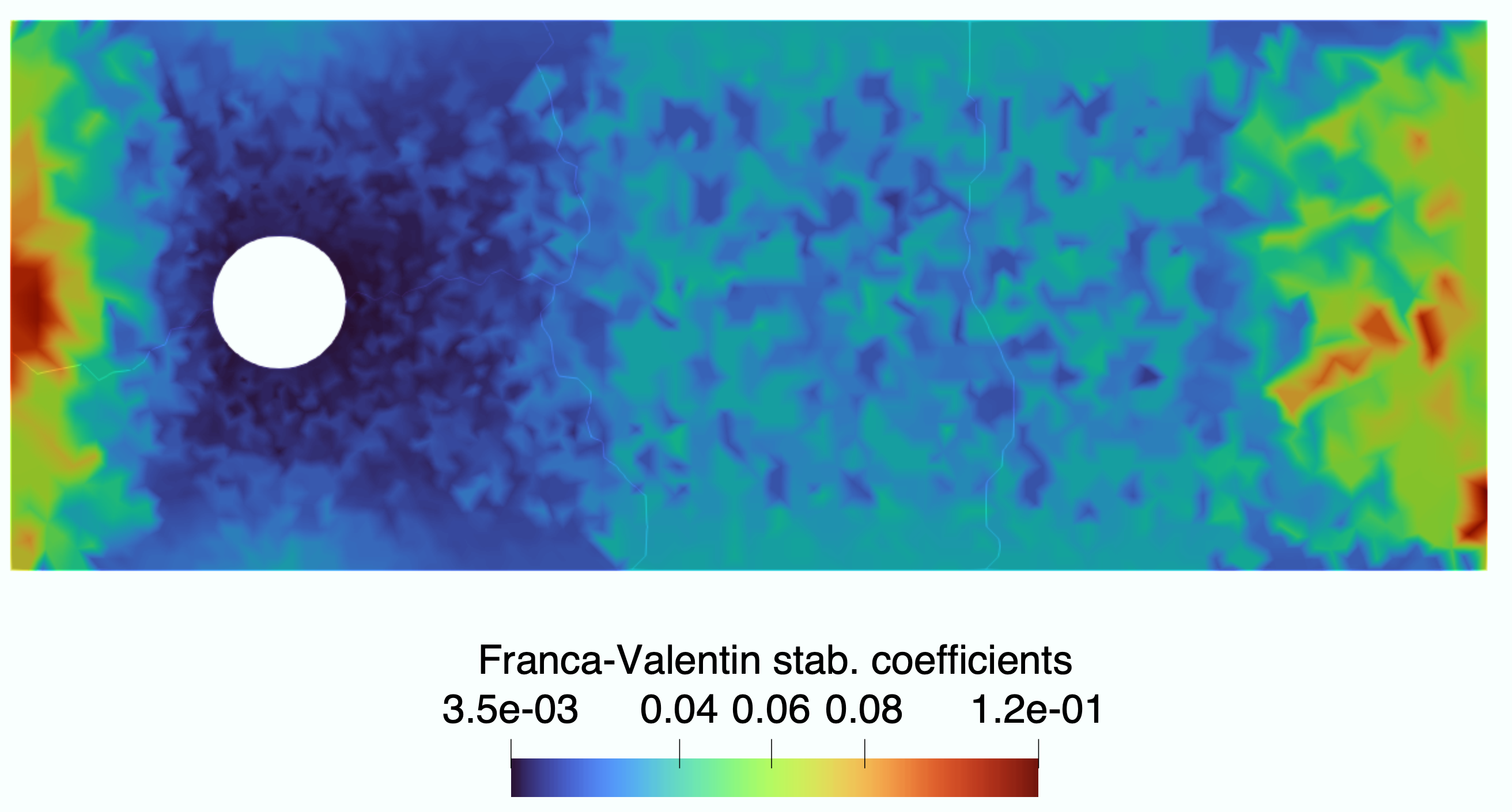}
\\
(c) VMS-Spectral stabilised coefficients &(d) Hauke stabilised coefficients\\
\includegraphics[width=0.5\linewidth]{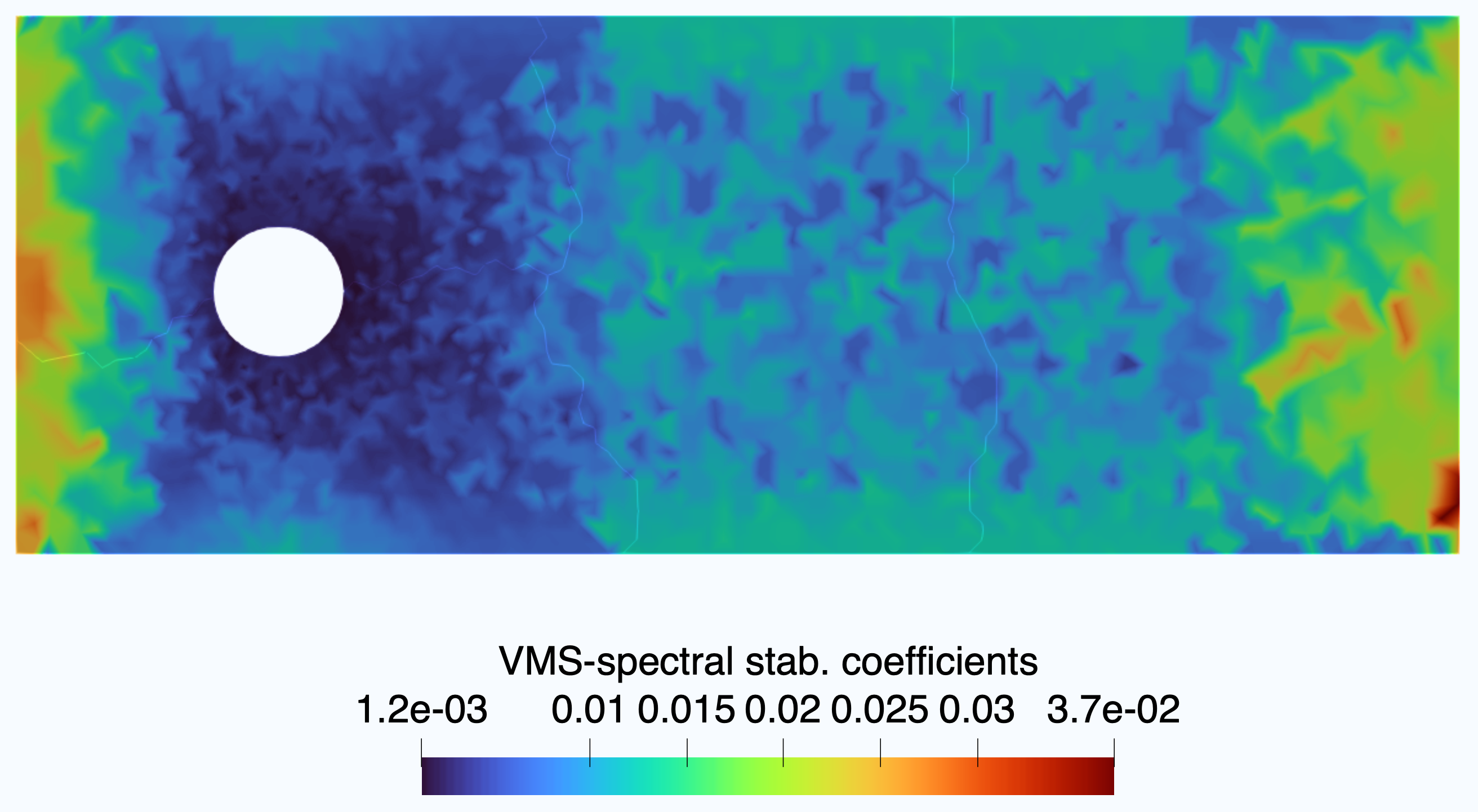}&\includegraphics[width=0.5\linewidth]{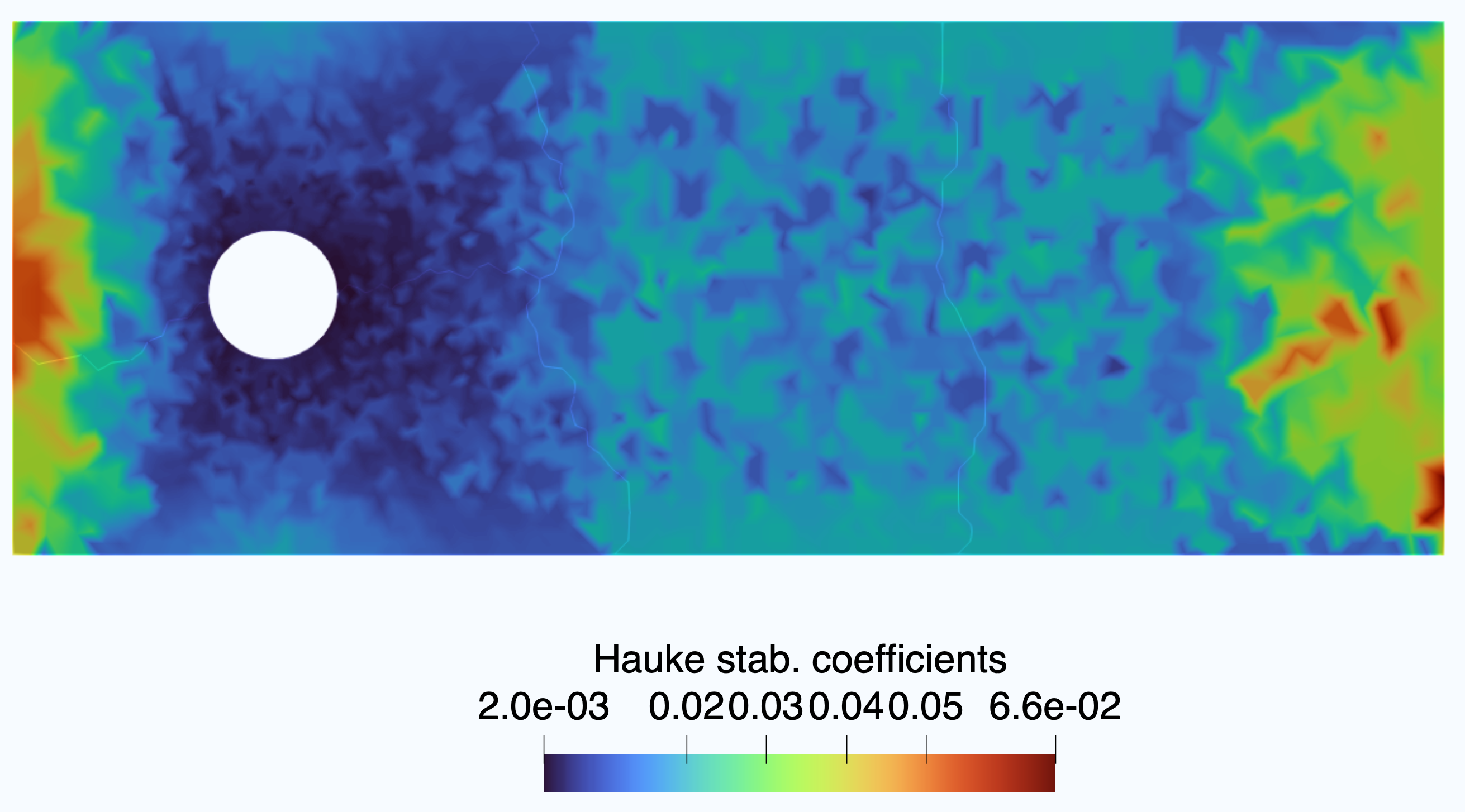}
\end{tabular}
\end{center}
\caption{\color{black}\label{fig_taus_cil} Test 3. Mappings of the stabilization coefficients: the least-squares stabilised coefficients, (panel (a)), {\color{black} Franca and Valentin stabilised coefficients (panel (b)), VMS-spectral stabilised coefficients (panel (c)) and  Hauke stabilised coefficients (panel (d)). Note that color scale in panels (a) and (b) is different from the one in (c) and (d).}}
\end{figure}

{\color{black} In Fig. \ref{fig_sol_LS}, we represent the solution obtained with the least-squares method for $\mu=0.0005$ and $\P_1$ f. e.. Finally, in Fig. \ref{fig_error_cil}, we represent the errors obtained with the least-squares stabilised coefficients , (panel (a)) and with Franca-Valentin stabilised coefficients, (panel (b)). The former are somewhat smaller, while the patterns of the values reached are similar but with some differences on the location of the highest values.}
\begin{figure}[ht!]
\begin{center}
\includegraphics[width=0.7\linewidth]{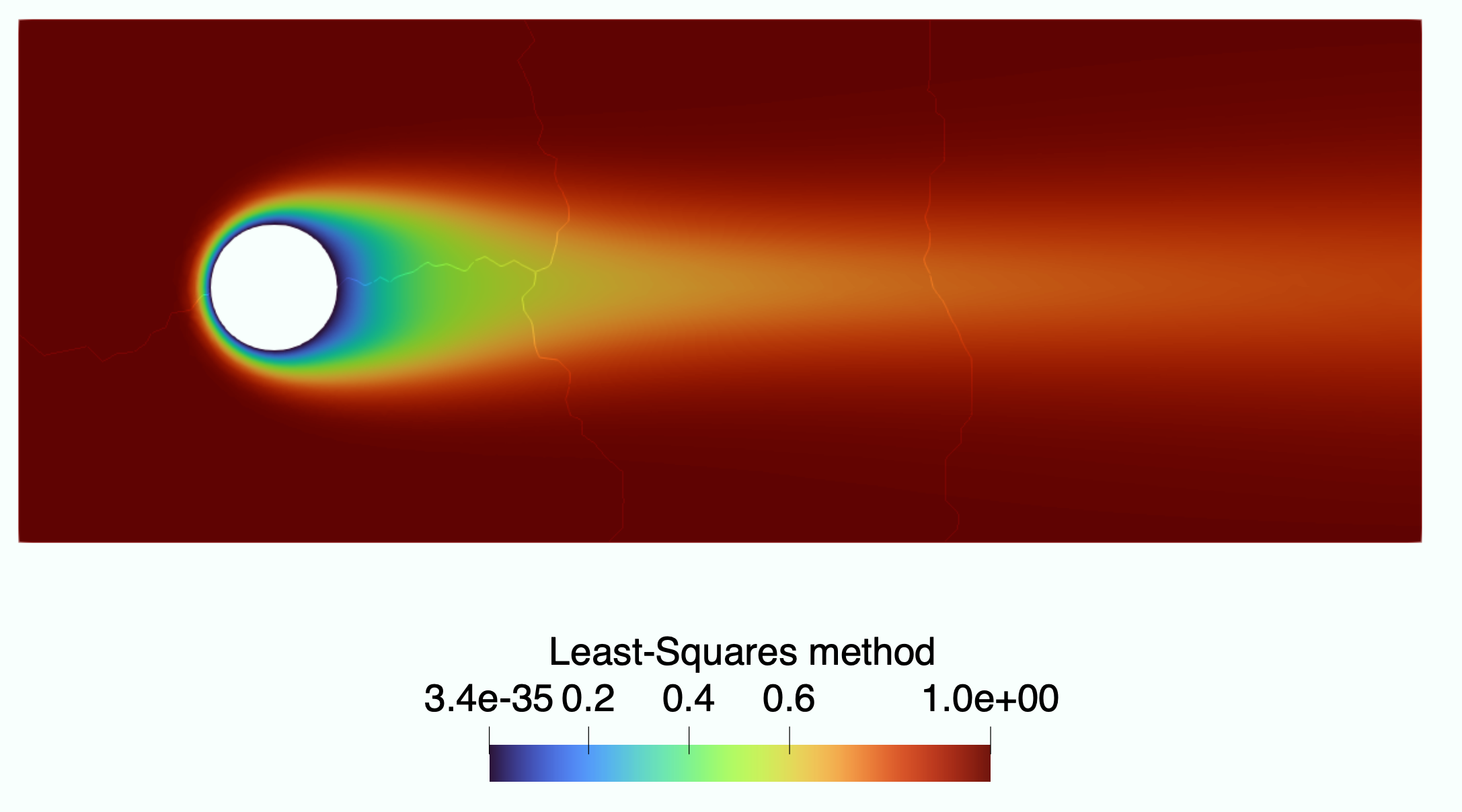}
\end{center}
\caption{\label{fig_sol_LS} Test 3. Representation of the solution obtained with the least-squares  method for $\mu=0.0005$ and $\P_1$ f. e..}
\end{figure}

\begin{figure}[ht!]
\begin{center}
\begin{tabular}{ll}(a) Error least-squares method&(b)  Franca-Valentin stabilised coefficients \\
\includegraphics[width=0.5\linewidth]{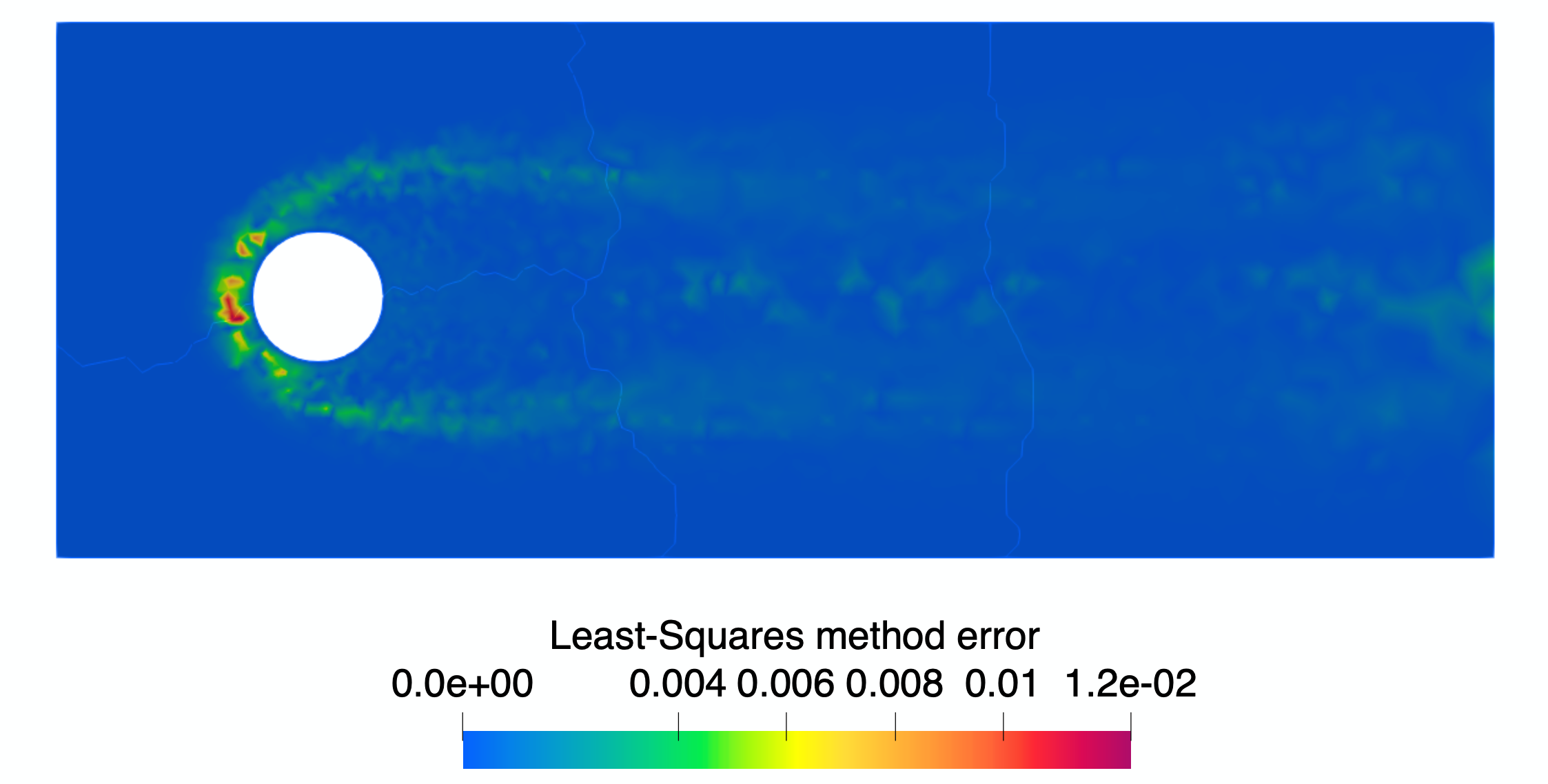}&\includegraphics[width=0.5\linewidth]{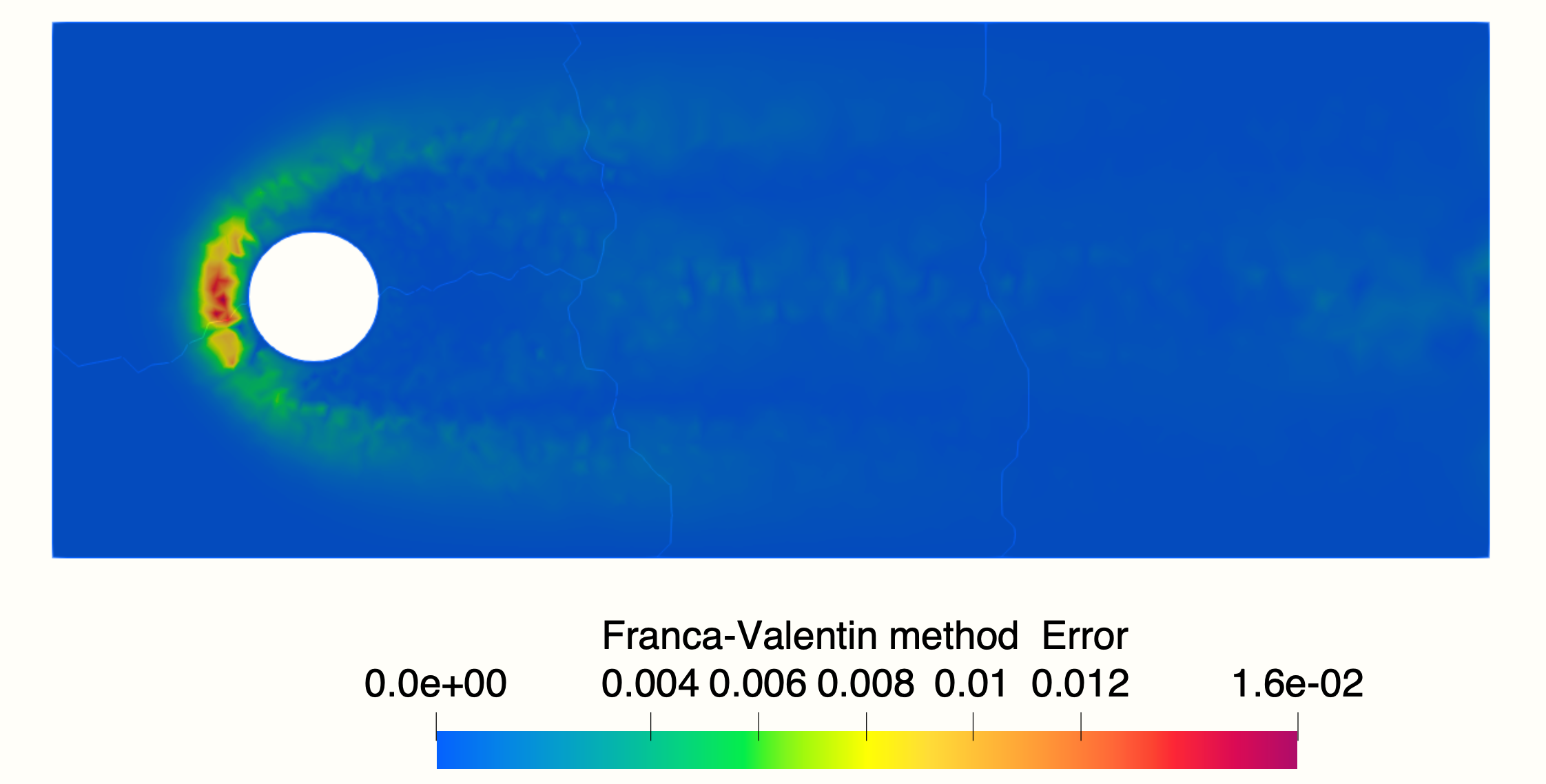}
\end{tabular}
\end{center}
\caption{\label{fig_error_cil} \color{black}Test 3. Mappings of the errors obtained with the VMS-spectral stabilised coefficients, (panel (a)) and with Codina's method of stabilised coefficients, (panel (b)). }
\end{figure}

 \subsection{Test 4: Lid-driven cavity flow}
 We test here the least-squares stabilised coefficients in a popular benchmark flow, the lid-driven cavity flow problem. 
In this test, the flow is modelled by the Navier-Stokes equations. It takes place in the unit square domain $\Omega=(0,1)\times(0,1)$ with Dirichlet boundary conditions,
\begin{equation}\label{NSE}
\left\{\begin{array}{l}
U\cdot\nabla U-\mu \Delta U+\nabla p=0\quad \mbox{in }\Omega,\\  \noalign{\smallskip}
U=0 \quad \mbox{on }\partial \Omega\setminus \gamma,\\ 
U=(1,0) \quad \mbox{on } \gamma,
\end{array}\right.
\end{equation}
where $\gamma$ is the top side of the unit square, $U$ is the velocity and $p$ is the pressure. This is a nonlinear problem that we solve through time stepping techniques to reach a stationary state. At each iteration step we consider  the linearised semi-discretisation in time
\begin{equation}\label{NSED}
\begin{array}{l}\displaystyle
\frac{U^{n+1}-U^n}{k}+ U^n\cdot\nabla U^{n+1}-\mu \Delta U^{n+1}+\nabla p^{n+1}=0,\quad \nabla \cdot U^n=0.
\end{array}
\end{equation}
where $k$ is the time step. In order to obtain inf-sup stable pairs of spaces,  we consider the finite element spaces $\mathbb{P}_1+$Bubble for the velocities and $\mathbb{P}_1$ for the pressure:
\begin{equation}\label{punobur}
X_h=[V_{0h}^{(1)} \oplus \mathbb{B}]^2, \,\,\mbox{with}\,\, \mathbb{B}=\{ v_h \in C^0(\bar{\Omega}) \,|\, {v_h}_{|_K} \in \mathbb{B}(K),\,\forall K\in{\cal T}_h,\,\, {v_h}_{|_{\partial \Omega}}=0\},
\end{equation}
where $\mathbb{B}(K)=\{\mathrm{span}(b_K),\,b_K=\lambda_{1K}\,\lambda_{2K}\,\lambda_{3K}\,\}$, $\lambda_{iK}$ being the barycentric coordinates of element $K \in {\cal T}_h$; and
$$
Q_h=\{ q_h \in V_h^{(1)}\,\,\mbox{such that}\,\, \int_\Omega q_h=0\},
$$
We either consider the Taylor-Hood finite element spaces,
$$
X_h=[V_{0h}^{(k)} ]^2, \quad Q_h=\{ q_h \in V_h^{(k-1)}\,\,\mbox{such that}\,\, \int_\Omega q_h=0\},
$$
for $k=2$ or $k=3$.
   We consider a steady-state-residual based least-squares stabilised discretisation of \eqref{NSED},\\
   
   Find $(U^{n+1}_h,p^{n+1}_h) \in X_h \times Q_h$ such that 
\begin{equation}\label{NSEDh}\nonumber
\begin{array}{c}
(U^{n+1}_h,V_h)+ k\, \left (U^n_h\cdot\nabla U^{n+1}_h,V_h\right )-k\,\mu \,(\nabla U^{n+1}_h, \nabla V_h)-k\, (p^{n+1}_h,\nabla \cdot V_h)+k \,(\nabla\cdot U^{n+1}_h,q_h) \\\displaystyle
+ k\,S_h(U_h^n;U^{n+1}_h,V_h)=(U^n_h,V_h),\qquad\quad\forall (V_h,q_h) \in X_h \times Q_h; \,\,\mbox{where}\end{array}
\end{equation} 
$$
S_h(U_h^n;U^{n+1}_h,V_h)=\sum_{K\in {\cal T}_h}\, \sum_{i=1}^2\,(\tau^{(K,n)}_i\,P_i(U_h^n;(U^{n+1}_h,p_h^{n+1})),P_i(U_h^n;(V_h,q_h))\, )_K,
$$
with
\begin{eqnarray*}
P(W;(U, p))&=& W\cdot\nabla U-\mu \Delta U+\nabla p; \end{eqnarray*}
and
$$
\tau^{(K,n)}_i= \tau_K(P_{i,K}^n), \,\, P_{i,K}^n=\frac{\|U_{i_{|_K}}^n\|_{0,K}\, h_K}{2\, \mu}.
$$
For the $\P_1$+Bubble finite element space $X_h$ defined in \eqref{punobur}, the stabilised coefficients have been computed following the procedure described in Section \ref{se:least} for the advection-diffusion equation, using the finite element space 
$$
V_h= V_{0h}^{(1)} \oplus \mathbb{B}.
$$

 The steady solution $(U_h,p_h)\in X_h \times Q_h$ eventually reached by this time stepping procedure satisfies $\forall (V_h,q_h) \in X_h \times Q_h$,
 \begin{equation}\label{NSEDh}\nonumber
\begin{array}{r}
 \left (U_h\cdot\nabla U_h,V_h\right )-\mu \,(\nabla U_h, \nabla V_h)- (p_h,\nabla \cdot V_h)+(\nabla\cdot U_h,q_h) + S_h(U_h;U_h,V_h)=0,
\end{array}
\end{equation} 
 By the standard theory of finite element approximation of Navier-Stokes equations, the sequence $(U_h,p_h)$ strongly converges to the solution $(U,p)$ of problem \eqref{NSE} in $[H^1(\Omega)]^2 \times L^2(\Omega)$, with  order $k$ in $h$ if $(U,p) $ belongs to $[H^{k+1}(\Omega)]^2 \times H^k(\Omega)$ (cf. \cite{Chacon}).
{\color{black}In Fig. \ref{fig_sol_lid}, we represent the exact solution in case $Re=1000$: velocity, (panel (a)) and pressure, (panel (b)).}

\begin{figure}[ht!] 
\begin{center}\color{black} 
\begin{tabular}{ll}(a) Velocity exact solution &(b)  Pressure exact solution\\
\includegraphics[width=0.5\linewidth]{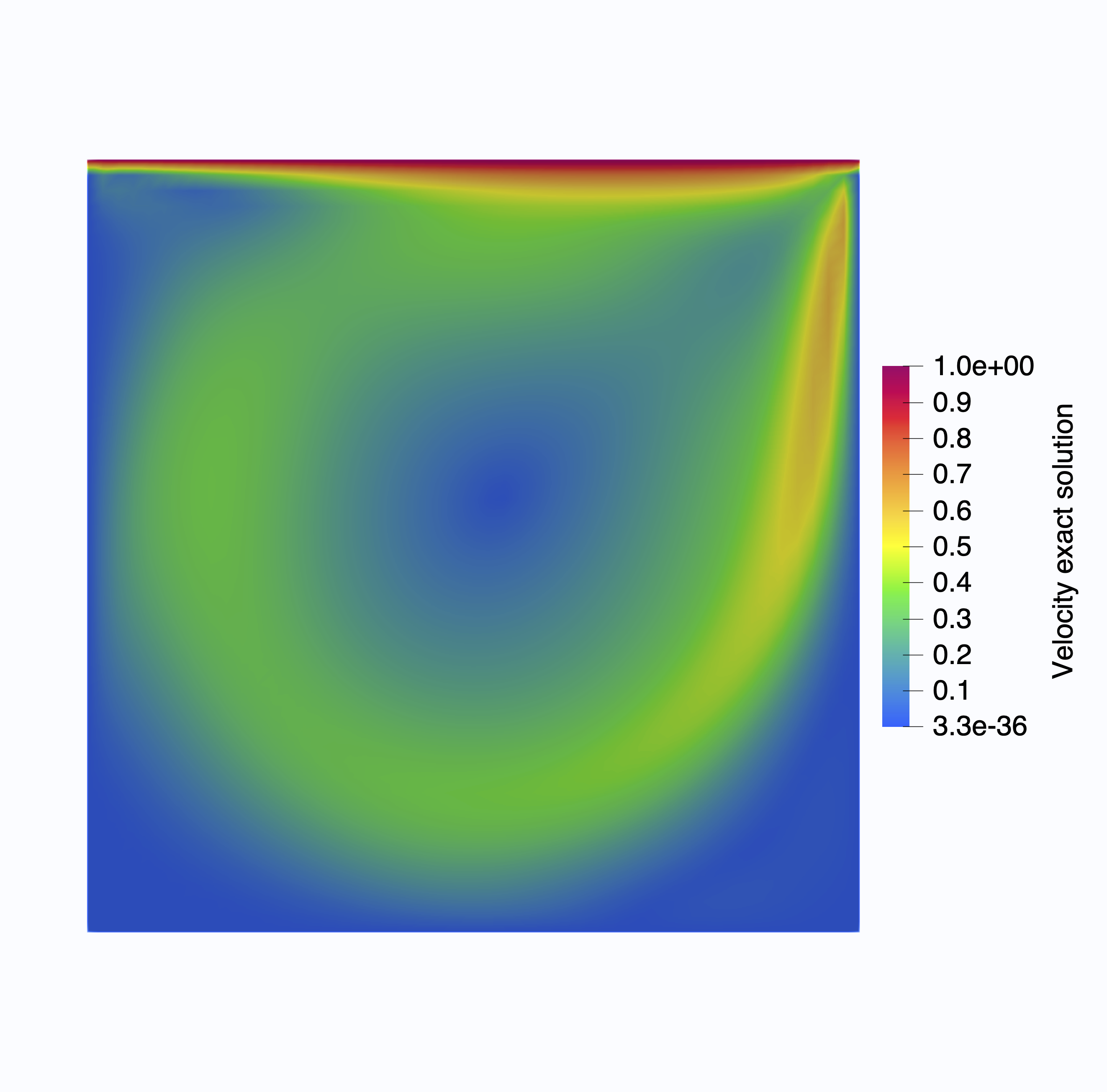}&\includegraphics[width=0.5\linewidth]{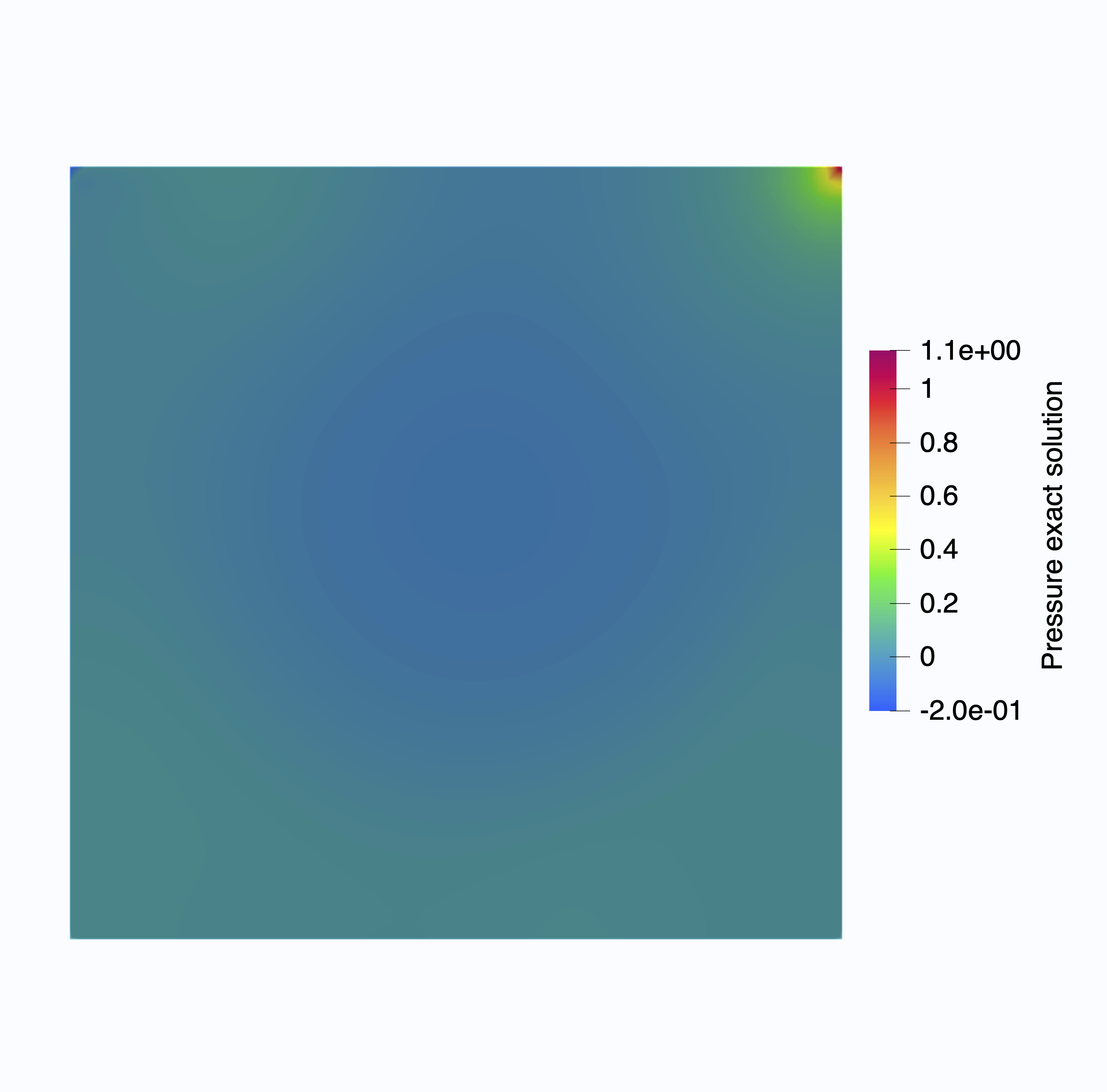}
\end{tabular}
\end{center}
\caption{\label{fig_sol_lid} \color{black} Test 4. Representation of the exact solution in case $Re=1000$: velocity, (panel (a)) and pressure, (panel (b)) with $\P_1-$bubble f. e. for velocity and $\P_1$ f. e. for pressure.} 
\end{figure}

 {\color{black} Table \ref{tab:tablecavidad} illustrates the obtained errors in $L^2$ norm for Reynold numbers $Re=1000$ and $Re=4000$, obtained from the quasi-stationary solution. These errors have been computed taking as reference solution the $(\P_1+\mbox{ Bubble},\P_1)$ Galerkin solution, the  $(\P_{2},\,\P_1)$ Galerkin solution, in a refined mesh of  40401 nodes, and the $(\P_{3},\,\P_2)$ Galerkin solution in a refined mesh of  32761 nodes.  The smallest errors are provided by varying stabilised coefficients, followed by the least-squares coefficients in all cases, with rather close error levels.  
 
  \begin{table}
 {\color{black}
  \begin{center}
   \begin{tabular}{l c c c c c c c}  
   \hline\hline                        
   $(\P_{1b},\,\P_1)$ & Re&${Re_h}-$range &$L^2_{LS}$ & $L^2_C$ & $L^2_H$ & $L^2_{flow}$  & $L^2_{VMS}$\\    [0.5ex]
   \hline                   
    & 1000 &(5.2e-05,12.80) & 0.279517 & 0.320112  & 0.289896 & 0.270412 &  \textbf{0.270231}\\ 
   & 4000 &(2.7e-03,47.47) &  0.370433 & 0.389394 &  0.388864 & \textbf{0.332643}  & 0.381904\\ 
   [0.5ex]        
\hline     
\hline\hline                        
     $(\P_{2},\,\P_1)$& Re&${Re_h}-$range &$L^2_{LS}$ & $L^2_C$ & $L^2_H$ & $L^2_{flow}$  \\    [0.5ex]
   \hline                   
    & 1000 &(9.3e-05,12.82) & 0.134368 & 0.135876  & \textbf{0.134272}& 0.134609 \\ 
   & 4000 &(8.7e-04,48.54) &  0.190527 & 0.192603 & 0.191508 & \textbf{0.189216}  \\ 
   [0.5ex]        
   \hline     
   \hline\hline                        
    $(\P_{3},\,\P_2)$ & Re&${Re_h}-$range &$L^2_{LS}$ & $L^2_C$ & $L^2_H$ & $L^2_{flow}$  \\    [0.5ex]
   \hline                   
    &1000 &(9.3e-05,12.82) & 0.122369 & 0.122248  &  0.122403& \textbf{0.12217} \\
   & 4000 &(6.9e-04,50.78) &  0.169775 & 0.170251 & 0.169832 & \textbf{0.16922}  \\ 
   [0.5ex]        
   \hline     
   \end{tabular}
    \end{center}
    \caption{\label{tab:tablecavidad} Test 4. Ranges of values of the grid Reynolds number ${Re_h}$, errors in $L^2$ norms of the stabilised solution obtained by the least-squares stabilised coefficients (with subindex $LS$), through Codina coefficients (with subindex $C$), through Hauke coefficients (with subindex $H$) and through Franca-Valentin coefficients (with subindex $flow$), for Reynolds numbers 1000 an 4000.
}
}
    \end{table}

{\color{black} In subsequent figures, we represent the results corresponding to $Re=1000$ and $(\P_{1b},\,\P_1)$ f. e.. In Fig. \ref{fig_taus_lid}, we represent the stabilisation coefficients obtained with the least-squares stabilised coefficients. Finally, in Fig. \ref{fig_error_lid}, we represent the errors obtained in the velocity in case $Re=1000$ with the  least-squares stabilised coefficients, (panel (a)), with Franca-Valentin stabilised coefficients, (panel (b)), with VMS-spectral  stabilised coefficients (panel (c)) and with Hauke stabilised coefficients (panel (d)). We observe quite similar patterns for all of them, with larger values in the flow areas with highest gradientes.}
%

\begin{figure}[ht!] \color{black} 
\begin{center}
\begin{tabular}{ll}
\includegraphics[width=0.5\linewidth]{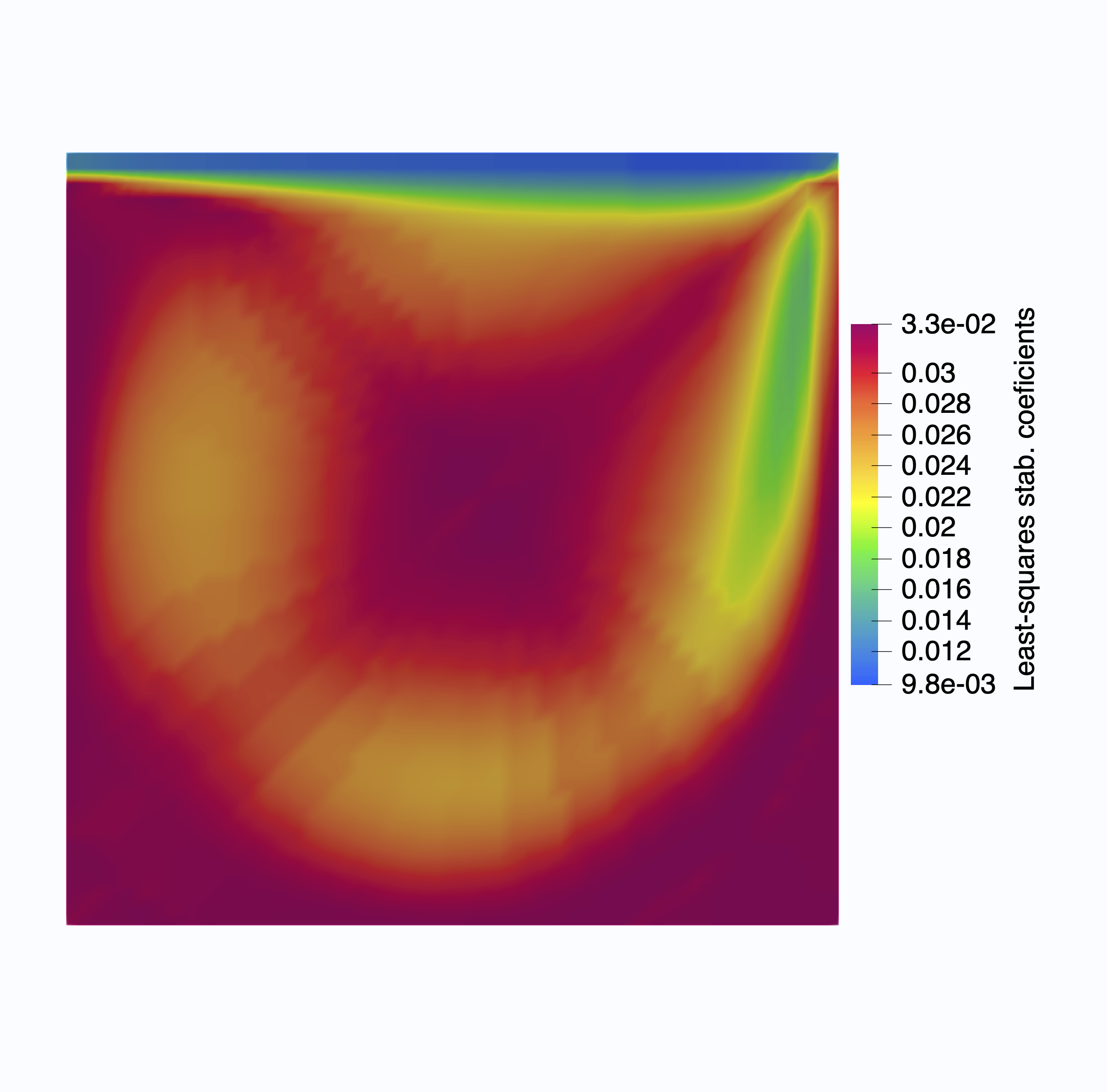}
\end{tabular}
\end{center}
\caption{\label{fig_taus_lid}\color{black} Test 4. Mapping of the least-squares stabilisation coefficients. }
\end{figure}

\begin{figure}[ht!]
\begin{center}\color{black} 
\begin{tabular}{ll}(a)   Error Least-squares &(b)  Error Franca-Valentin\\
\includegraphics[width=0.5\linewidth]{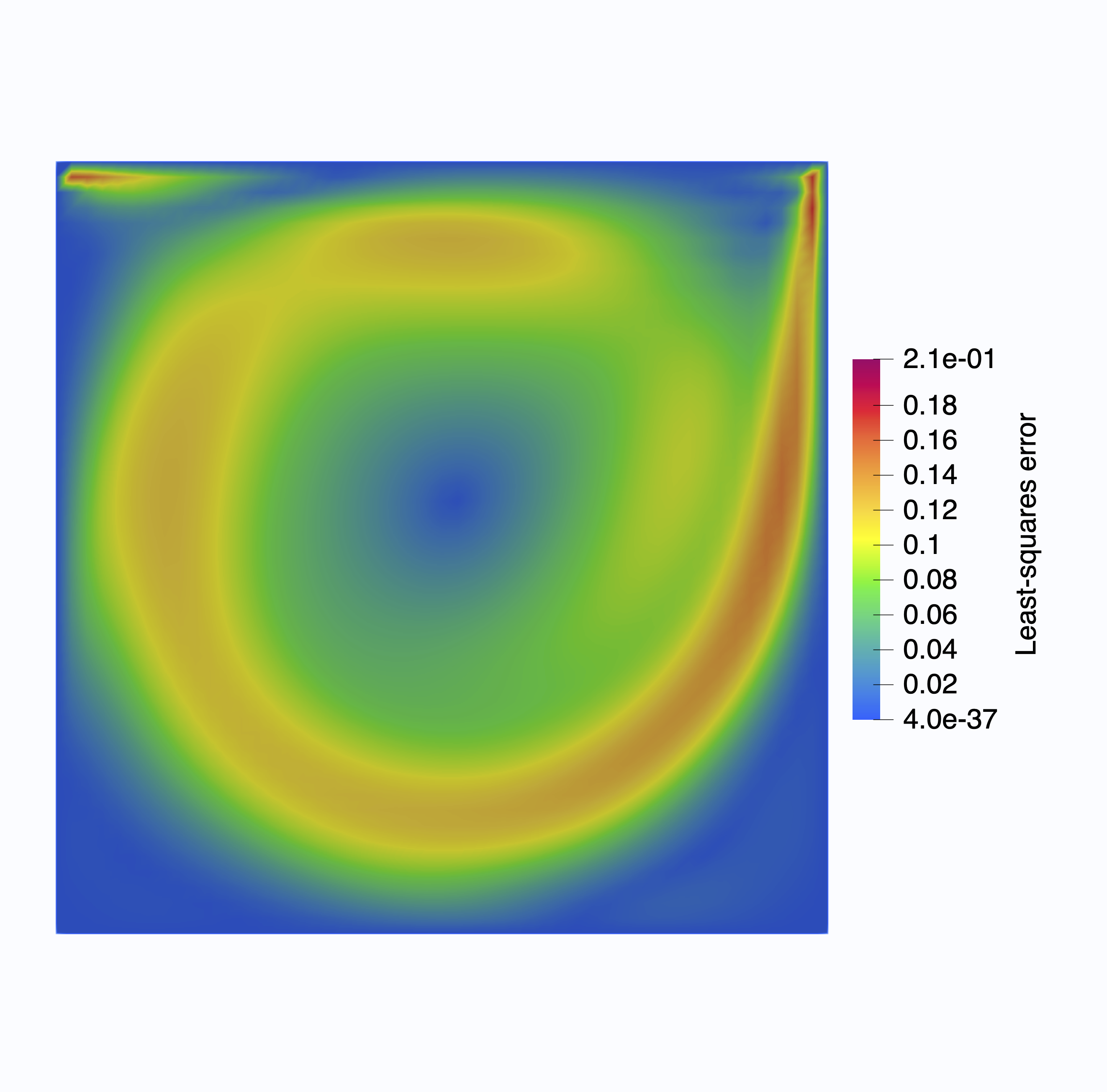}&\includegraphics[width=0.5\linewidth]{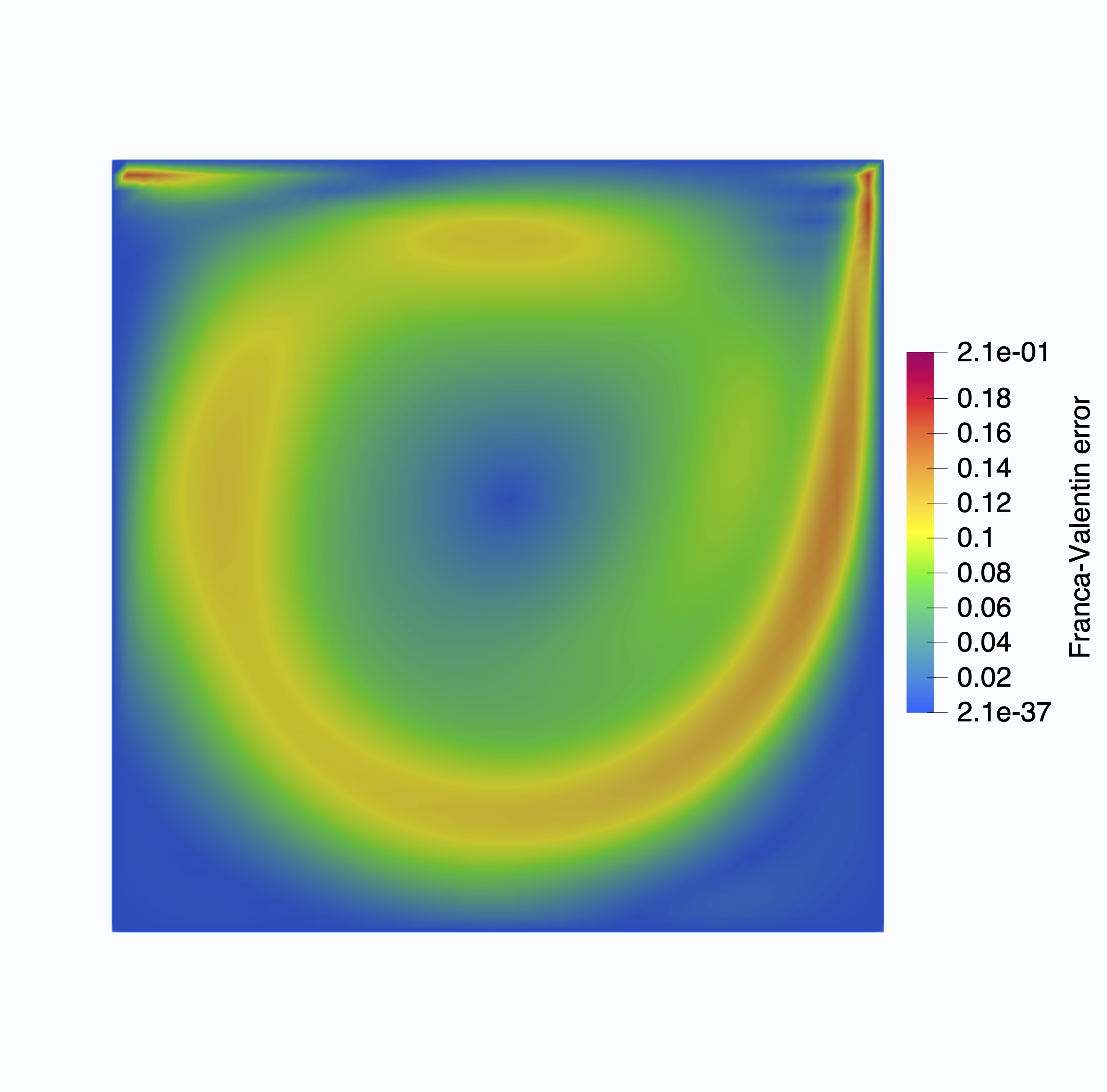}\\
{(c)  Error VMS-spectral } &{(d)  Error Hauke}\\
\includegraphics[width=0.5\linewidth]{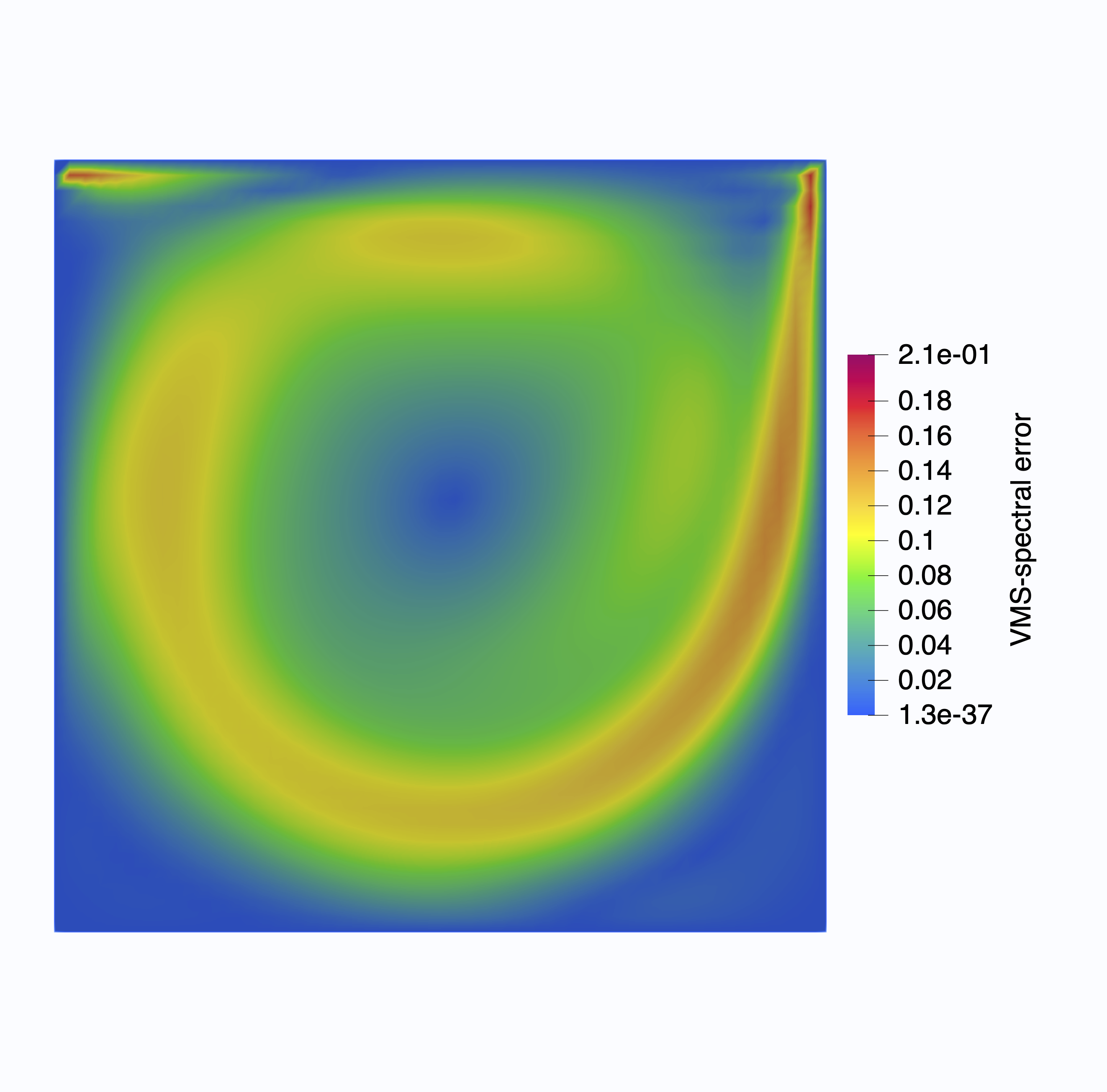}&\includegraphics[width=0.5\linewidth]{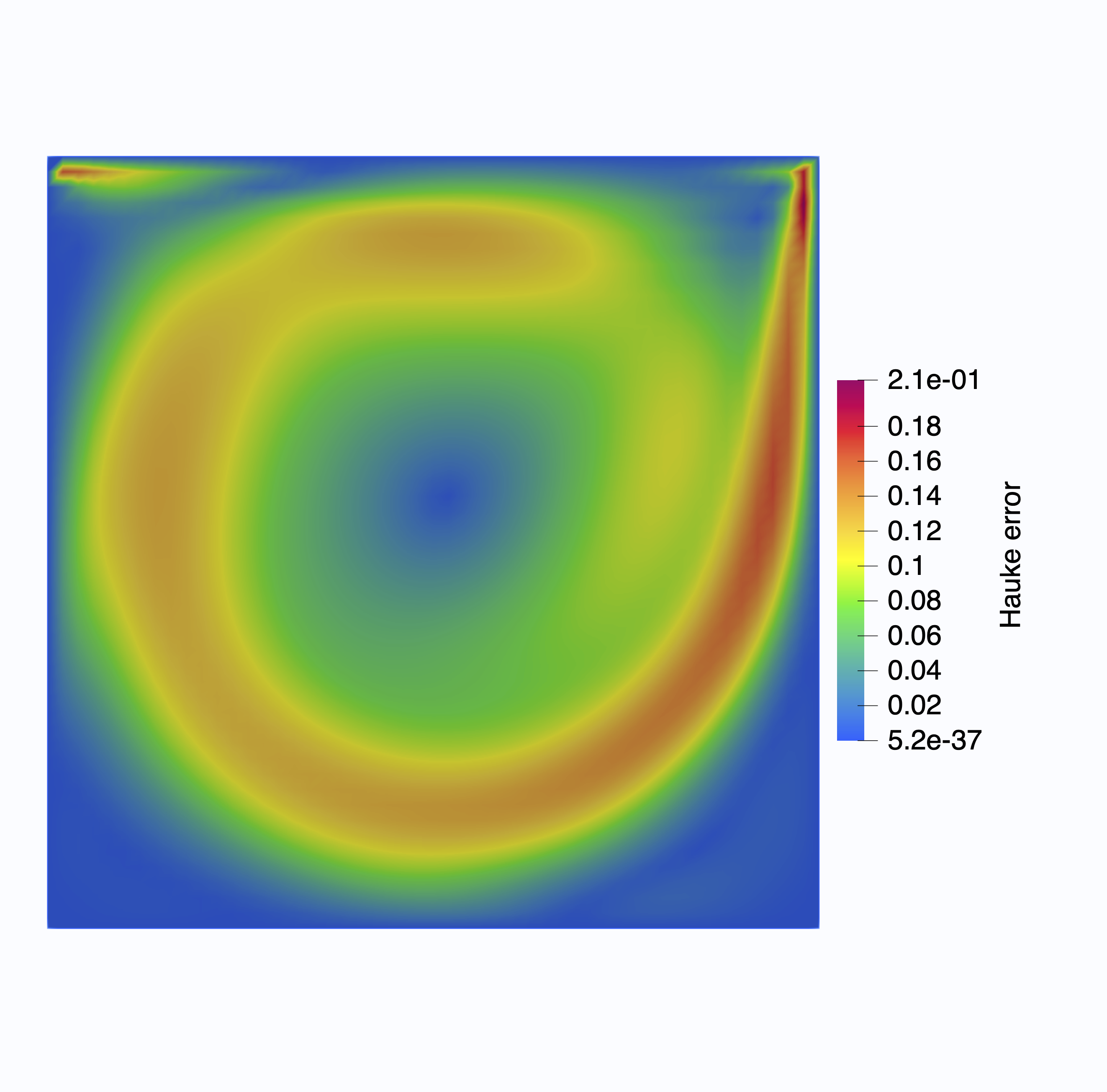}
\end{tabular}
\end{center}
\caption{\label{fig_error_lid} \color{black} Test 4. Mappings of the errors obtained in the velocity solution in case $Re=1000$ with the least-squares stabilised coefficients, (panel (a)), with Franca-Valentin stabilised coefficients, (panel (b)), with VMS-spectral stabilised coefficients (panel (c)) and with Hauke stabilised coefficients (panel (d)).}
\end{figure}
 
\section{Conclusions and perspectives}\label{se:conclusions}
In this paper we have assessed the solution of incompressible flow equations by means of stabilised methods, by introducing least-squares computed stabilised coefficients. We have stated that these can be efficiently computed as minima of smooth convex functionals. We have also introduced a data-driven off-line/on-line strategy to compute them in the flow simulation process with low computational cost. In the off-line strategy the stabilised coefficients are computed as functions of the non-dimensional parameters that govern the flow at grid element level.

We have compared the errors provided by the least-squares stabilised coefficients to those provided by several previously established stabilised coefficients with several advection-diffusion and Navier-Stokes flows, considering isotropic and anisotropic advection velocities, as well as isotropic and anisotropic grids, and $\P_1$, $\P_1$+Bubble, $\P_2$ and $\P_3$ finite elements. 

We observe that in all tested flows the least-squares stabilised coefficients provide nearly the smallest errors, in any case staying very close to the smallest ones.

In addition, the least-squares procedure to compute the stabilised coefficients has the advantage to apply to any finite element or finite volume discretisation, as well as to more general (compressible, multi-phase, thermal, ...) flows.  In despite of its need of a rather large amount of computation in the off-line procedure, it is thus a rewarding procedure, worth to be applied to general stabilised solutions of flow problems.
\section*{Acknowledgements}
This research is partially supported by Junta de Andaluc\'{\i}a - FEDER Fund Programa Operativo FEDER Andaluc\'{\i}a 2014-2020 grant US-1254587.

\end{document}